\documentclass[reqno,11pt]{article}
\usepackage{graphicx, amsmath, amsthm}
\usepackage{amssymb}

\input{stochbif.sty}
\def\Ac{E\setminus A}
\def\gammabar{\bar\gamma}
\DeclareMathOperator*\osc{osc}
\def\lOu{\lambda_0^{\rm u}}
\def\lOs{\lambda_0^{\rm s}}
\def\loneu{\lambda_1^{\rm u}}
\def\lones{\lambda_1^{\rm s}}
\def\piOu{\pi_0^{\rm u}}
\def\piOs{\pi_0^{\rm s}}
\def\Ku{K^{\rm u}}
\def\Ks{K^{\rm s}}
\def\ku{k^{\rm u}}
\def\logs{\abs{\log\sigma}}
\def\hper{h^{\text{\rm{per}}}}
\def\hpertilde{\tilde{h}^{\text{\rm{per}}}}

\begin{document}


\title{On the noise-induced passage\\through an unstable periodic orbit II:\\ 
General case}
\author{Nils Berglund\thanks{Supported by ANR project MANDy, Mathematical
Analysis of Neuronal Dynamics, ANR-09-BLAN-0008-01.} 
~and Barbara Gentz\thanks{Supported by CRC 701 \lq\lq Spectral Structures and Topological Methods in Mathematics\rq\rq}}
\date{}   

\maketitle

\begin{abstract}
\noindent
Consider a dynamical system given by a planar differential equation, which
exhibits an unstable periodic orbit surrounding a stable periodic orbit. It is
known that under random perturbations, the distribution of  locations where the
system's first exit from the interior of the unstable orbit occurs, typically 
displays the phenomenon of cycling: The distribution of first-exit locations is
translated along the unstable periodic orbit proportionally to the logarithm of
the noise intensity as the noise intensity goes to zero. We show that for a
large class of such systems, the cycling profile is given, up to a
model-dependent change of coordinates, by a universal function given by a
periodicised Gumbel distribution. Our techniques combine action-functional or
large-deviation results with properties of random Poincar\'e maps
described by continuous-space discrete-time Markov chains. 
\end{abstract}

\leftline{\small{\it Date.\/} 
August 13, 2012. Revised, July 24, 2013. 
}
\leftline{\small 2010 {\it Mathematical Subject Classification.\/} 
60H10, 		
34F05   	
(primary), 
60J05,   	
60F10   	
(secondary)}
\noindent{\small{\it Keywords and phrases.\/}
Stochastic exit problem, 
diffusion exit,  
first-exit time, 
characteristic
boundary, 
limit cycle, 
large deviations, 
synchronization, 
phase slip, 
cycling, 
stochastic resonance,
Gumbel distribution.}  


\section{Introduction}
\label{sec_in}

Many interesting effects of noise on deterministic dynamical systems can be
expressed as a stochastic exit problem. Given a subset $\cD$ of phase space,
usually assumed to be positively invariant under the deterministic flow, the
stochastic exit problem consists in determining when and where the noise causes
solutions to leave $\cD$. 

If the deterministic flow points inward $\cD$ on the boundary $\partial\cD$,
then the theory of large deviations provides useful answers to the exit
problem in the limit of small noise intensity~\cite{FW}. Typically, the exit
locations are concentrated in one or several points, in which the so-called
quasipotential is minimal. The mean exit time is exponentially long as a function of the noise intensity, and the
distribution of exit times is asymptotically exponential~\cite{Day1}. 

The situation is more complicated when $\partial\cD$, or some part of it, is
invariant under the deterministic flow. Then the theory of large deviations
does not suffice to characterise the distribution of exit locations. An
important particular case is the one of a two-dimensional deterministic
ordinary differential equation (ODE), admitting an unstable periodic orbit. Let
$\cD$ be the part of the plane inside the periodic orbit. Day~\cite{Day5,Day7}
discovered a striking phenomenon called~\emph{cycling}: As the noise intensity
$\sigma$ goes to zero, the exit distribution rotates around the boundary
$\partial\cD$, by an angle proportional to $\abs{\log\sigma}$. Thus the exit
distribution does not converge as $\sigma\to0$. The phenomenon of cycling has
been further analysed in several works by Day~\cite{Day3,Day6,Day4}, by Maier
and Stein~\cite{MS4,MS1}, and by Getfert and
Reimann~\cite{Getfert_Reimann_2009,Getfert_Reimann_2010}.

The noise-induced exit through an unstable periodic orbit has many important
applications. For instance, in synchronisation it determines the distribution of
noise-induced phase slips \cite{PRK}. The first-exit distribution also
determines the residence-time
distribution in stochastic resonance~\cite{GHM,MS2,BG9}. In neuroscience, the
interspike interval statistics of spiking neurons is described by a
stochastic exit problem~\cite{Tuckwell75,Tuckwell,BG_neuro09,BerglundLandon}. In
certain cases, as
for the Morris--Lecar model~\cite{MorrisLecar81} for
a region of parameter values, the spiking mechanism involves the passage through an
unstable periodic orbit (see,
e.g.~\cite{RinzelErmentrout,Tateno_Pakdaman_2004,
Tsumoto_etal_2006,Ditlevsen_Greenwood_12}). 
In all these cases, it is important to know the distribution of first-exit
locations as precisely as possible. 

In~\cite{BG7}, we introduced a simplified model, consisting of two linearised
systems patched together by a switching mechanism, for which we obtained an explicit expression for the
exit distribution. In appropriate coordinates, the distribution has the form of
a periodicised Gumbel distribution, which is common in extreme-value theory. 
Note that the standard Gumbel distribution also occurs in the description of
reaction paths for overdamped Langevin
dynamics~\cite{CerouGuyaderLelievreMalrieu12}. 
The aim of the present work is to generalise the results of~\cite{BG7} to a
larger class of more realistic systems. Two important ingredients of the
analysis are large-deviation estimates near the unstable periodic orbit, and
the theory of continuous-space Markov chains describing random Poincar\'e maps. 

The remainder of this paper is organised as follows. In Section~\ref{sec_res},
we define the system under study, discuss the heuristics of its behaviour, 
state the main result (Theorem~\ref{main_theorem}) and discuss its consequences.
Subsequent sections are devoted to the proof of this result. 
Section~\ref{sec_cs} describes a coordinate transformation to polar-type
coordinates used throughout the analysis. Section~\ref{sec_ld} contains the
large-deviation estimates for the dynamics near the unstable orbit.
Section~\ref{sec_mc} states results on Markov chains and random Poincar\'e
maps, while Section~\ref{sec_sp} contains estimates on the sample-path behaviour
needed to apply the results on Markov chains. Finally, in Section~\ref{sec_el}
we complete the proof of Theorem~\ref{main_theorem}. 

\subsubsection*{Acknowledgement} 
We would like to thank the referees for their careful reading of the first
version of this manuscript, and for their constructive suggestions, which
led to improvements of the main result as well as the presentation.



\section{Results}
\label{sec_res}


\subsection{Stochastic differential equations with an unstable periodic orbit}
\label{ssec_reso}

Consider the two-dimensional deterministic ODE 
\begin{equation}
\label{csd1}
\dot z = f(z)\;,
\end{equation}
where $f\in\cC^2(\cD_0,\R^2)$ for some open, connected set
$\cD_0\subset\R^2$. We assume that this system admits two distinct periodic
orbits, that is, there are periodic functions $\gamma_\pm: \R \to \R^2$, of
respective periods $T_\pm$, such that 
\begin{equation}
\label{csd2}
\dot\gamma_\pm(t) = f(\gamma_\pm(t)) 
\qquad\qquad
\forall t\in\R\;.
\end{equation}
We set $\Gamma_\pm(\ph) = \gamma_\pm(T_\pm\ph)$, so that
$\ph\in\fS^1=\R/\Z$ gives an equal-time parametrisation of the orbits.
Indeed, 
\begin{equation}
\label{csd3}
\dtot{}{\ph} \Gamma_\pm(\ph) = T_\pm f(\Gamma_\pm(\ph))\;,
\end{equation}
and thus $\dot\ph=1/T_\pm$ is constant on the periodic orbits. 

Concerning the geometry, we will assume that the orbit $\Gamma_-$ is
contained in the interior of $\Gamma_+$, and that the annulus-shaped region
$\cS$ between the two orbits contains no invariant proper subset. This
implies in particular that the orbit through any point in $\cS$ approaches
one of the orbits $\Gamma_\pm$ as $t\to\infty$ and $t\to-\infty$. 

\begin{figure}
\centerline{\includegraphics*[clip=true,height=50mm]{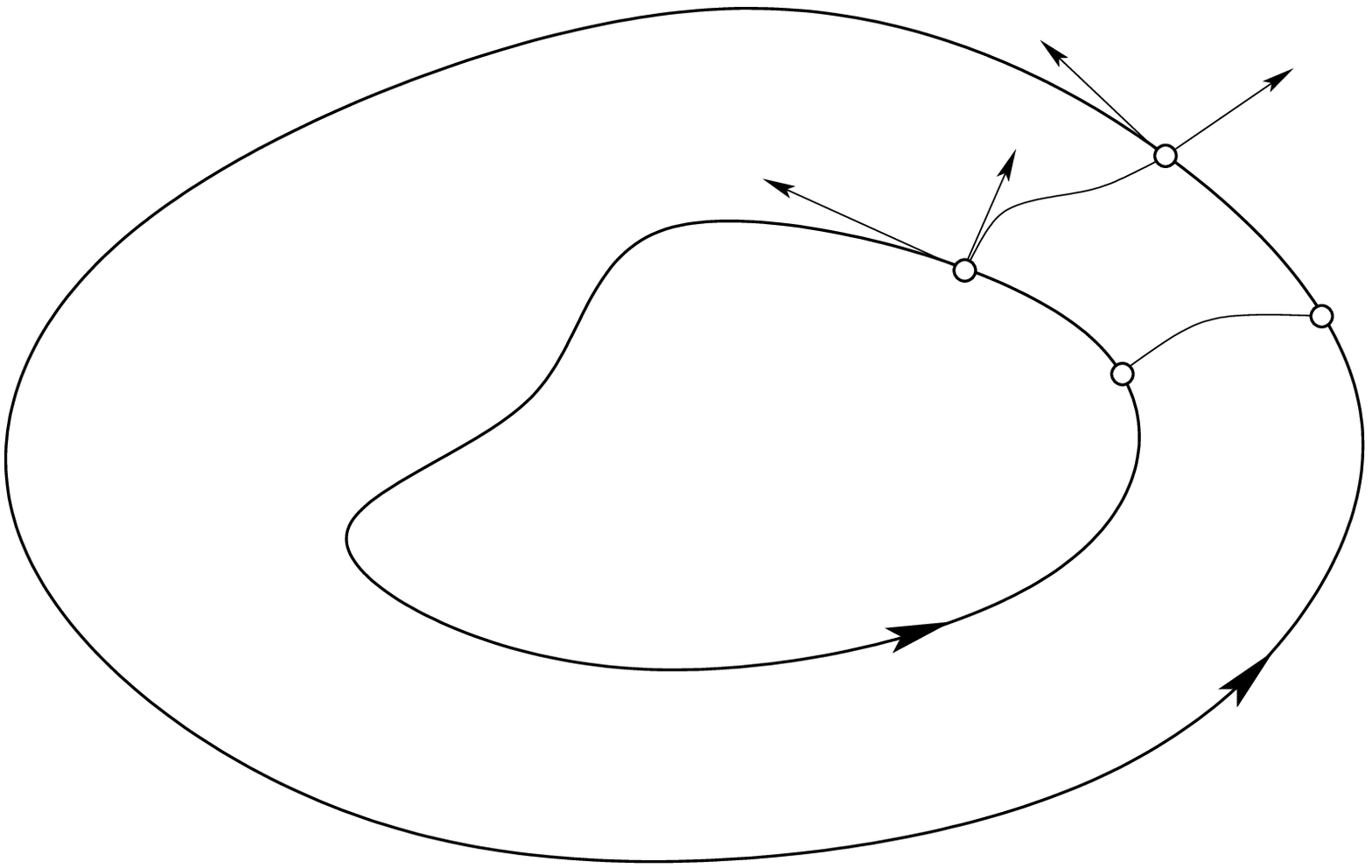}}
 \figtext{
	\writefig	8.9	3.1	$\Gamma_-(0)$
	\writefig	11.3	3.6	$\Gamma_+(0)$
	\writefig	10.4	3.3	$\Delta_0$
	\writefig	8.0	3.6	$\Gamma_-(\ph)$
	\writefig	10.4	4.5	$\Gamma_+(\ph)$
	\writefig	9.6	4.1	$\Delta_\ph$
	\writefig	8.3	4.7	$u_-(\ph)$
	\writefig	11.0	5.0	$u_+(\ph)$
	\writefig	6.4	4.5	$f(\Gamma_-(\ph))$
	\writefig	9.5	5.4	$f(\Gamma_+(\ph))$
	\writefig	5.4	3.0	$\gamma_-$
	\writefig	3.3	3.9	$\gamma_+$
 }
 \caption[]
 {Geometry of the periodic orbits. The stable orbit $\gamma_-$ is located
inside the unstable orbit $\gamma_+$. $\Gamma_\pm(\ph)$ denote parametrisations
of the orbits, and $u_\pm(\ph)$ are eigenvectors of the monodromy matrix used to
construct a set of polar-type coordinates.}
\label{fig_csd}
\end{figure}

Let $A_\pm(\ph) = \sdpar fz(\Gamma_\pm(\ph))$ denote the Jacobian
matrices of $f$ at $\Gamma_\pm(\ph)$. The principal solutions associated with
the linearisation around the periodic orbits 
are defined by 
\begin{equation}
\label{csd4}
\sdpar U\ph _\pm(\ph,\ph_0) = 
T_\pm A_\pm(\ph) U_\pm(\ph,\ph_0)\;, 
\qquad\qquad
U_\pm(\ph_0,\ph_0) = \one\;.
\end{equation}
In particular, the monodromy matrices $U_\pm(\ph+1,\ph)$ satisfy 
\begin{equation}
\det U_\pm(\ph+1,\ph) 
= \exp\biggset{T_\pm \int_\ph^{\ph+1} \Tr A_\pm(\ph')\,\6\ph'}
\;,
\label{csd5}
\end{equation} 
with $\Tr A_\pm(\ph') = \divergence f(\Gamma_\pm(\ph'))$. 
Taking the derivative of~\eqref{csd3} shows that  each monodromy matrix
$U_\pm(\ph+1,\ph)$ admits $f(\Gamma_\pm(\ph))$ as eigenvector with
eigenvalue $1$. The other eigenvalue is thus also independent of $\ph$,
and we denote it $\e^{\pm\lambda_\pm T_\pm}$, where 
\begin{equation}
\label{csd6}
\pm\lambda_\pm = \int_0^1 \divergence f(\Gamma_\pm(\ph))\,\6\ph
\end{equation}
are the Lyapunov exponents of the orbits. We assume that $\lambda_+$ and
$\lambda_-$ are both positive, which implies that $\Gamma_-$ is stable and
$\Gamma_+$ is unstable. The products $\lambda_\pm T_\pm$ have the
following geometric interpretation: a small ball centred in
the stable periodic orbit will shrink by a factor $\e^{-\lambda_-T_-}$
at each revolution around the orbit, while a small ball centred in the
unstable orbit will be magnified by a factor $\e^{\lambda_+T_+}$.

Consider now the stochastic differential equation (SDE) 
\begin{equation}
\label{css1}
\6z_t = f(z_t)\,\6t + \sigma g(z_t)\6W_t\;,
\end{equation}
where $f$ satisfies the same assumptions as before, $\set{W_t}_t$ is a
$k$-dimensional standard Brownian motion, $k\geqs 2$, and
$g\in\cC^1(\cD_0,\R^{2\times k})$
satisfies the uniform ellipticity condition 
\begin{equation}
 \label{ellipticity}
 c_1 \norm{\xi^2} 
 \leqs \pscal{\xi}{g(z)\transpose{g(z)}\xi}
 \leqs c_2 \norm{\xi^2}
 \qquad
 \forall z\in\cD_0\ \forall \xi\in\R^2
\end{equation}
with $c_2\geqs c_1>0$. 

\begin{prop}[Polar-type coordinates]
There exist $L>1$ and a set of coordinates $(r,\ph)\in(-L,L)\times\R$, in which
the SDE~\eqref{css1} takes the form 
\begin{align}
\nonumber
\6r_t &= f_r(r_t,\ph_t;\sigma) \,\6t + \sigma g_r(r_t,\ph_t) \6W_t\;,\\
\6\ph_t &= f_\ph(r_t,\ph_t;\sigma) \,\6t 
+ \sigma g_\ph(r_t,\ph_t) \6W_t\;.
\label{polar_coords}
\end{align}
The functions $f_r,f_\ph,g_r$ and $g_\ph$ are periodic with period $1$ in
$\ph$, and $g_r,g_\ph$ satisfy a uniform ellipticity condition similar
to~\eqref{ellipticity}. 
The unstable orbit lies in $r=1+\Order{\sigma^2}$, and 
\begin{align}
\nonumber
f_r(r,\ph) &= \lambda_+(r-1) + \Order{(r-1)^2} \;,\\
f_\ph(r,\ph) &= \frac1{T_+} + \Order{(r-1)^2}
\label{polar1}
\end{align}
as $r\to1$. The stable orbit lies in $r=-1+\Order{\sigma^2}$, and 
\begin{align}
\nonumber
f_r(r,\ph) &= -\lambda_-(r+1) + \Order{(r+1)^2} \;,\\
f_\ph(r,\ph) &= \frac1{T_-} + \Order{(r+1)^2}
\label{polar2}
\end{align}
as $r\to-1$. Furthermore, $f_\ph$ is strictly larger than a positive constant
for all $(r,\ph)\in(-L,L)\times\R$, and $f_r$ is negative for $-1<r<1$. 
\end{prop}

We give the proof in Section~\ref{sec_cs}. We emphasize that after performing this change of coordinates, the stable and
unstable orbit are not located exactly in $r=\pm1$, but are slightly shifted
by an amount of order $\sigma^2$, owing to second-order terms in It\^o's
formula. 

\begin{remark}
\label{remark_coordinates} 
The system of coordinates $(r,\ph)$ is not unique. However, it is characterised
by the fact that the drift term near the periodic orbits is as simple as
possible. Indeed, $f_\ph$ is constant on each periodic orbit (equal-time
parametrisation), and $f_r$ does not depend on $\ph$ to linear order near the
orbits. These properties will be preserved if we apply shifts to $\ph$ (which
may be different on the two periodic orbits), and if we locally scale the
radial variable $r$. The construction of the change of variables shows that its
nonlinear part interpolating between the the orbits is quite arbitrary, but we
will see that this does not affect the results to leading order. 

It would be possible to further simplify the diffusion terms on the periodic
orbits $g_r(\pm 1,\ph)$, preserving the same structure of the equations, by
combining $\ph$-dependent transformations which are linear near the orbits with
a random time change (see  Section~\ref{ssec_resmr}). However this would
introduce other technical difficulties that we want to avoid. 
\end{remark}

The question we are interested in is the following: Assume the system starts with some
initial condition $(r_0,\ph_0)=(r_0,0)$ close to the stable periodic orbit. What is the
distribution of the first-hitting location of the unstable orbit? We define the first-hitting time of (a $\sigma^2$-neighbourhood of) the unstable orbit by
\begin{equation}
 \label{exit_times}
 \tau = \inf\bigsetsuch{t>0}{r_t = 1} \;,
\end{equation} 
 so that the random variable $\ph_\tau$ gives the first-exit location.
 Note that we consider $\ph$ as belonging to $\R_+$ instead of the circle
$\R/\Z$, which means that we keep track of the number of rotations around the
periodic orbits. 


\subsection{Heuristics 1: Large deviations}
\label{ssec_resld}

A first key ingredient to the understanding of the distribution of exit
locations is the theory of large deviations, which has been developed in the
context of SDEs by Freidlin and Wentzell~\cite{FW}. The theory tells us that
for a set $\Gamma$ of paths $\gamma:[0,T]\to\R^2$, one has 
\begin{equation}
\label{ld5}
-\inf_{\Gamma^\circ}I 
\leqs \liminf_{\sigma\to0} \sigma^2\log\bigprob{(z_{t})_{t\in[0,T]}\in\Gamma}
\leqs \limsup_{\sigma\to0} \sigma^2\log\bigprob{(z_{t})_{t\in[0,T]}\in\Gamma}
\leqs -\inf_{\overline\Gamma}I\;,
\end{equation}
where the \emph{rate function}\/ $I=I_{[0,T]}:\cC^0([0,T],\R^2)\to\R_+$ is given
by 
\begin{equation}
\label{ld6}
I(\gamma) = 
\begin{cases}
\dfrac12 \displaystyle\int_0^T \transpose{(\dot\gamma_s - f(\gamma_s))}
D(\gamma_s)^{-1} (\dot\gamma_s - f(\gamma_s)) \,\6s
& \text{if $\gamma\in H^1$,}\\
+\infty
& \text{otherwise,}
\end{cases}
\end{equation}
with $D(z)=g(z)\transpose{g(z)}$ (the \emph{diffusion matrix}, with components
$D_{rr}, D_{r\ph}=D_{\ph r}, D_{\ph\ph}$).
Roughly speaking, Equation~\eqref{ld5} tells us that
\begin{equation}
 \label{ld7a}
 \bigprob{(z_{t})_{t\in[0,T]}\in\Gamma} \simeq \e^{-\inf_\Gamma I/\sigma^2}
 \quad \ \text{or, symbolically,}\quad\ 
 \bigprob{(z_{t})_{t\in[0,T]}=\gamma} \simeq \e^{-I(\gamma)/\sigma^2}\;.
\end{equation} 
For deterministic solutions, we have $\dot\gamma=f(\gamma)$ and $I(\gamma)=0$, so
that~\eqref{ld7a} does not yield useful information. However, for paths
$\gamma$ with $I(\gamma)>0$,~\eqref{ld7a} tells us how unlikely $\gamma$ is. 

The minimisers of $I$ obey Euler--Lagrange equations, which are equivalent
to Hamilton equations generated by the Hamiltonian 
\begin{equation}
\label{ld7}
H(\gamma,\psi) = \frac12 \transpose\psi D(\gamma) \psi +
\transpose{f(\gamma)}\psi\;,
\end{equation} 
where $\psi = D(\gamma)^{-1} (\dot\gamma - f(\gamma))$ is the moment
conjugated to $\gamma$. The rate function thus takes the form 
\begin{equation}
\label{ld8}
I(\gamma) = \dfrac12 \int_0^T \transpose{\psi_s} D(\gamma_s) \psi_s
\,\6s\;.
\end{equation}
Writing $\transpose{\psi}=(p_r,p_\ph)$, the Hamilton equations associated
with~\eqref{ld7}
read 
\begin{equation}
\label{ld9}
\begin{split}
\dot r &= f_r(r,\ph) + D_{rr}(r,\ph) p_r + D_{r\ph}(r,\ph) p_\ph\;, \\
\dot \ph &= f_\ph(r,\ph) + D_{r\ph}(r,\ph) p_r + D_{\ph\ph}(r,\ph) p_\ph\;, \\
\dot p_r &= - \sdpar{f_r}r(r,\ph) p_r - \sdpar{f_\ph}r(r,\ph) p_\ph 
- \frac12 \sum_{ij\in\set{r,\ph}} \sdpar{D_{ij}}r (r,\ph) p_i p_j\;, \\
\dot p_\ph &= - \sdpar{f_r}\ph(r,\ph) p_r - \sdpar{f_\ph}\ph(r,\ph) p_\ph 
- \frac12 \sum_{ij\in\set{r,\ph}} \sdpar{D_{ij}}\ph (r,\ph) p_i p_j\;, \\
\end{split}
\end{equation}
We can immediately note the following points:
\begin{itemiz}
\item	the plane $p_r=p_\ph=0$ is invariant, it corresponds to the
deterministic dynamics;
\item	there are two periodic orbits, given by $p_r=p_\ph=0$ and $r=\pm1$,
which are, of course, the original periodic orbits of the deterministic
system;
\item	$\dot \ph$ is positive, bounded away from zero, in a neighbourhood of
the deterministic manifold. 
\end{itemiz}

The Hamiltonian being a constant of the motion, the four-dimensional phase
space is foliated in three-dimensional invariant manifolds, which can be
labelled by the value of $H$. Since $\sdpar{H}{p_r}=\dot r$ is positive near
the deterministic manifold, one can express $p_\ph$ as a function of $H$, $r$,
$\ph$ and $p_r$, and thus describe the dynamics on each invariant manifold
by an effective three-dimensional equation for $(r,\ph,p_r)$. It is
furthermore possible to use $\ph$ as new time, which yields a two-dimensional,
non-autonomous equation.\footnote{The associated Hamiltonian is the function
$P_\ph(r,p_r,H,\ph)$ obtained by expressing $p_\ph$ as a function of the other
variables.}

\begin{figure}
\centerline{\includegraphics*[clip=true,height=70mm]{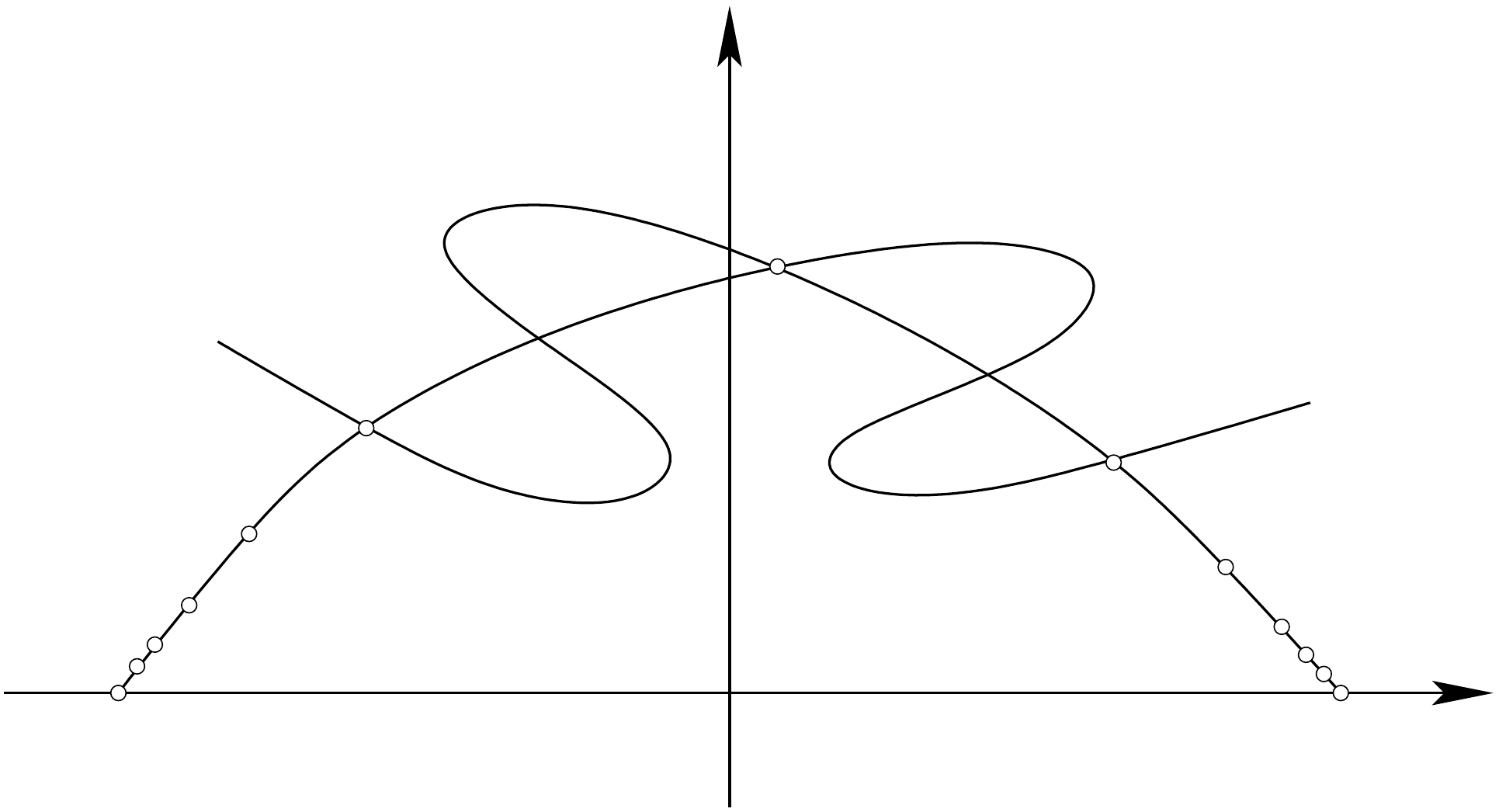}}
 \figtext{
	\writefig	13.2	1.1	$r$
	\writefig	12.4	1.1	$1$
	\writefig	1.6	1.1	$-1$
	\writefig	6.8	6.6	$p_r$
	\writefig	7.4	5.5	$z^*_0$
	\writefig	10.3	3.8	$z^*_1$
	\writefig	11.5	2.8	$z^*_2$
	\writefig	12.0	2.2	$z^*_3$
	\writefig	3.8	4.1	$z^*_{-1}$
	\writefig	2.4	3.0	$z^*_{-2}$
	\writefig	1.8	2.3	$z^*_{-3}$
	\writefig	8.8	5.6	$\cW^{\text{u}}_-$
	\writefig	5.3	5.95	$\cW^{\text{s}}_+$
	\writefig	8.8	1.65	$\cW^{\text{u}}_+$
	\writefig	5.1	1.65	$\cW^{\text{s}}_-$
 }
 \caption[]
 {Poincar\'e section of the Hamiltonian flow associated with the
large-deviation rate function. The stable periodic orbit is located in
$(-1,0)$, the unstable one in $(1,0)$. We assume that the unstable manifold
$\cW^{\text{u}}_-$ of $(-1,0)$ intersects the stable manifold
$\cW^{\text{s}}_+$ of $(1,0)$ transversally. The intersections of both
manifolds define a heteroclinic orbit $\set{z^*_k}_{-\infty<k<\infty}$ which
corresponds to the minimiser of the rate function. }
\label{fig_heteroclinic}
\end{figure}

The linearisation of the system around the periodic orbits is given by 
\begin{equation}
\label{ld10}
\dtot{}{\ph}
\begin{pmatrix}
r \\ p_r
\end{pmatrix}
=
\begin{pmatrix}
\pm\lambda_\pm T_\pm  & D_{rr}(\pm1,\ph) \\
0 & \mp\lambda_\pm  T_\pm
\end{pmatrix}
\begin{pmatrix}
r \\ p_r
\end{pmatrix}
\;.
\end{equation}
The characteristic exponents of the periodic orbit in $r=1$ are thus $\pm
T_+\lambda_+$, and those of the periodic orbit in $r=-1$ are $\pm
T_-\lambda_-$. The Poincar\'e section at $\ph=0$ will thus have hyperbolic
fixed points at $(r,p_r)=(\pm1,0)$.  

Consider now the event $\Gamma$ that the
stochastic system, starting on the stable orbit at $\ph=0$, hits the unstable
orbit for the first time near $\ph=s$. The probability of $\Gamma$ will be
determined by the infimum of the rate function $I$ over all paths connecting
$(r,\ph)=(-1,0)$ to $(r,\ph)=(1,s)$. Note however that if $t>s$, we can connect
$(1,s)$ to $(1,t)$ for free in terms of the rate function $I$ by following the deterministic dynamics along the
unstable orbit. We conclude that \emph{on the level of large deviations, all
exit points on the unstable orbit are equally likely}. 

This does not mean, however, that all paths connecting the stable and unstable
orbits are optimal. In fact, it turns out that the infimum of the rate function
is reached on a heteroclinic orbit connecting the orbits in infinite time. It
is possible to connect the orbits in finite time, at the cost of increasing the
rate function. In what follows, we will make the following
simplifying assumption.

\begin{assump}
\label{assump_heteroclinic} 
In the Poincar\'e section for $H=0$, the unstable manifold $\cW^{\text{u}}_-$ of
$(-1,0)$ intersects the stable manifold $\cW^{\text{s}}_+$ of $(1,0)$
transversally (\figref{fig_heteroclinic}). Let $\gamma_\infty$ denote the
heteroclinic orbit meeting the Poincar\'e section at the set
$\set{z^*_k}_{-\infty<k<\infty}$ of intersections of the manifolds. Then
$\gamma_\infty$ minimises the rate function over all paths connecting the two
periodic orbits, and this minimiser is unique (up to translations
$\ph\mapsto\ph+1$).  
\end{assump}

This assumption obviously fails to hold if the system is perfectly rotation
symmetric, because then the two manifolds do not intersect transversally but are
in fact identical.  The assumption is likely to be true \emph{generically} for
small-amplitude perturbations of $\ph$-independent systems (cf.\ Melnikov's
method), for large periods $T_\pm$ (adiabatic limit) and for small periods
(averaging regime), but may not hold in general. See in
particular~\cite{GrahamTel84,GrahamTel85,MS1} for discussions of possible
complications. 

It will turn out in our analysis that the probability of crossing the unstable
orbit near a sufficiently large finite value of $\ph$ will be determined by a
finite number $n=\intpart{\ph}$ of translates of the minimising orbit. 


\subsection{Heuristics 2: Random Poincar\'e maps}
\label{ssec_resrp}

The second key ingredient of our analysis are Markov chains describing
Poincar\'e maps of the stochastic system. Choose an initial condition
$(R_0,0)$, and consider the value $R_1=r_{\tau_1}$ of $r$ at the time 
\begin{equation}
 \label{resrp1}
 \tau_1 = \inf\setsuch{t>0}{\ph_t=1}\;,
\end{equation} 
when the sample path first reaches $\ph=1$ (\figref{fig_Markov}).
Since we are interested in the first-passage time through the unstable orbit, we
declare that whenever the sample path $(r_t,\ph_t)$ reaches $r=1$ before
$\ph=1$, then $R_1$ has reached a cemetery state $\partial$, which it never
leaves again. Successively, we define $R_n = r_{\tau_n}$, 
where $\tau_n = \inf\setsuch{t>0}{\ph_t=n}$, 
$n\in\N$.

\begin{figure}
\centerline{\includegraphics*[clip=true,height=60mm]{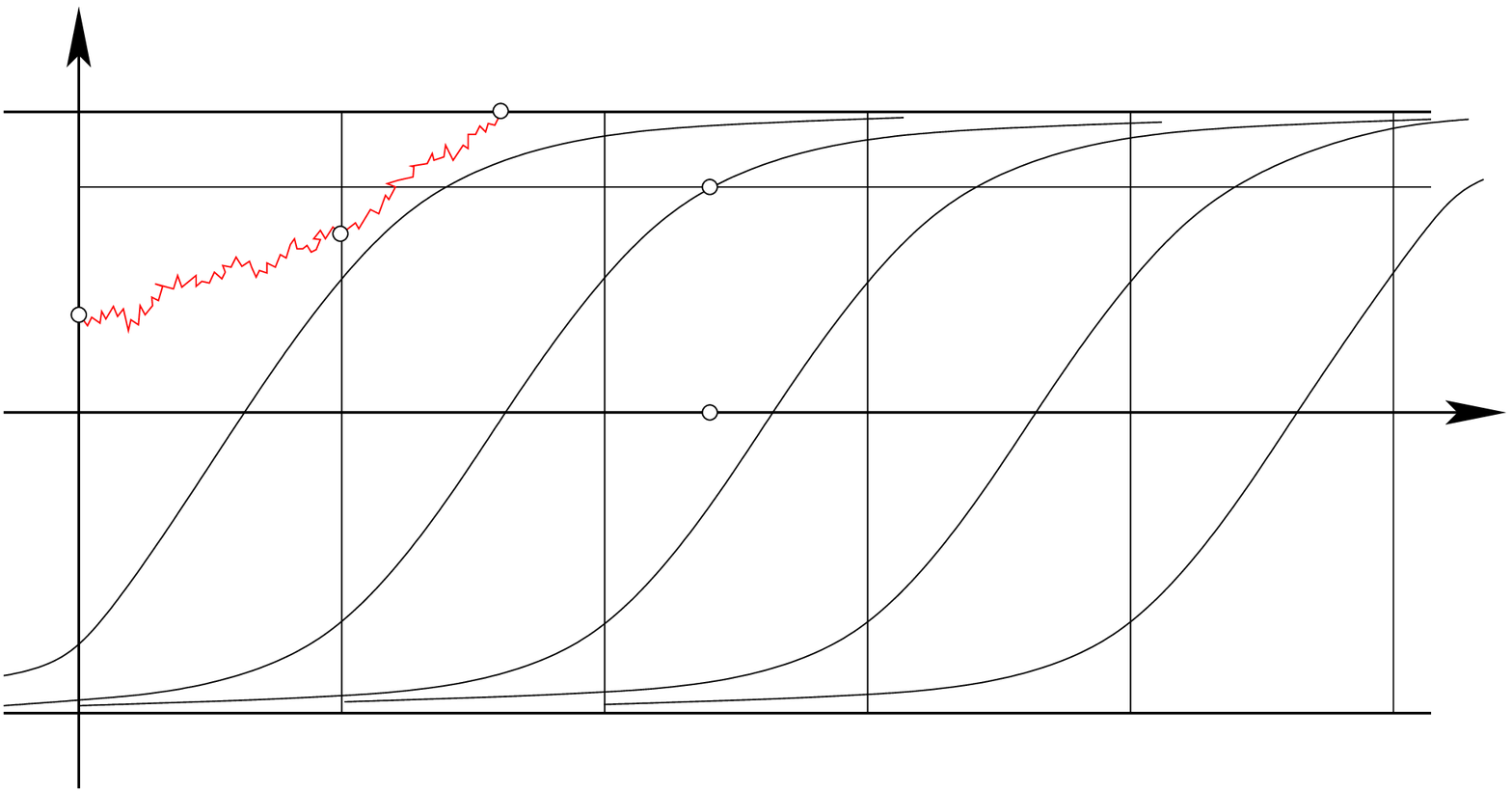}}
 \figtext{
	\writefig	12.4	3.5	$\ph$
	\writefig	1.9	5.9	$r$
	\writefig	1.3	5.0	$1-\delta$
	\writefig	1.7	4.0	$R_0$
	\writefig	3.7	4.8	$R_1$
	\writefig	4.0	3.05	$1$
	\writefig	6.9	3.5	$s^*$
	\writefig	5.1	2.55	$\gamma_\infty$
	\writefig	5.2	5.9	$\ph_\tau$
 }
 \caption[]
 {The optimal path $\gamma_\infty$ minimising the rate function, and its
translates. We define a random Poincar\'e map, giving the location $R_1$ of the
first crossing of the line $\ph=1$ of a path starting in $r=R_0$ and $\ph=0$.}
\label{fig_Markov}
\end{figure}

By periodicity of the system in $\ph$ and the strong Markov property, the
sequence $(R_0,R_1,\dots)$ forms a Markov chain, with kernel $K$, that is, 
for a Borel set $A$, 
\begin{equation}
 \label{resrp2}
 \bigprobin{R_n}{R_{n+1}\in A} = K(R_n,A) = \int_A K(R_n,\6y)
 \qquad
 \text{for all $n\geqs0$\;,}
\end{equation} 
where $ \bigprobin{x}{R_{n }\in A}$ denotes the probability that the Markov chain, starting in $x$, is in $A$ at time $n$.

Results on harmonic measures~\cite{BenArous_Kusuoka_Stroock_1984} imply that
$K(x,\6y)$ actually has a density $k(x,y)$ with respect to Lebesgue measure (see
also \cite{Dahlberg1977,JerisonKenig82,Cranston_Zhao_87} for related results).
Thus the density of $R_n$ evolves according to an integral operator with kernel
$k$. Such operators have been studied, among others, by
Fredholm~\cite{Fredholm_1903}, Jentzsch~\cite{Jentzsch1912} and
Birkhoff~\cite{Birkhoff1957}. In particular, we know that $k$ has a discrete set
of eigenvalues $\lambda_0, \lambda_1, \dots$ of finite multiplicity, where
$\lambda_0$ is simple, real, positive, and larger than the modules of all other
eigenvalues. It is called the~\emph{principal eigenvalue} of the Markov chain.
In our case, we have $\lambda_0<1$ due to the killing at the unstable orbit. 

Fredholm theory yields a decomposition\footnote{If $\lambda_1$ has multiplicity
$m>1$, the second term in~\eqref{resrl3} has to be replaced by a sum with $m$
terms.}
\begin{equation}
 \label{resrl3}
 k(x,y) = \lambda_0 h_0(x) h_0^*(y) + \lambda_1 h_1(x) h_1^*(y) + \dots
\end{equation} 
where the $h_i$ and $h^*_i$ are right and left orthonormal eigenfunctions of the
integral operator. It is known that $h_0$ and $h_0^*$ are positive and
real-valued~\cite{Jentzsch1912}. It follows that 
\begin{equation}
 \label{resrp4}
 \bigprobin{R_0}{R_n\in A} \bydef K_n(R_0,A) 
 = \lambda_0^n h_0(R_0) \int_A h_0^*(y)\6y 
  \biggbrak{1 + \biggOrder{\biggpar{\frac{\abs{\lambda_1}}{\lambda_0}}^n}}\;.
\end{equation} 
Thus the \emph{spectral gap\/}
$\lambda_0 - \abs{\lambda_1}$ plays an important role in the convergence of the distribution of $R_{n}$.
For times $n$ satisfying
$n\gg (\log(\lambda_0/\abs{\lambda_1}))^{-1}$, the distribution of $R_n$ will
have a density 
proportional to $h^*_0$. More precisely, if 
\begin{equation}
 \label{resrp5}
 \pi_0(\6x) = \frac{h^*_0(x)\6x}{\int h^*_0(y)\6y}
\end{equation} 
is the so-called \defwd{quasistationary distribution (QSD)\footnote{See for
instance~\cite{Yaglom56,Seneta_VereJones_1966}. A general bibliography on QSDs
by Phil Pollett is available at 
{\tt http://www.maths.uq.edu.au/$\sim$pkp/papers/qsds/}.}},
then the asymptotic distribution of the process $R_n$, conditioned on survival,
will be $\pi_0$, while the survival probability decays like $\lambda_0^n$. 

Furthermore, the (sub-)probability density of the first-exit location $\ph_\tau$ at
$n+s$, with $n\in \N$ and $s\in[0,1)$, can be written as
\begin{equation}
 \label{resrp6}
 \int K_{n}(R_0,\6y)\bigprobin{y}{\ph_\tau\in\6s}
 = \lambda_0^n h_0(R_0) \int h^*_0(y) \bigprobin{y}{\ph_\tau\in\6s} \6y 
 \biggbrak{1 + \biggOrder{\biggpar{\frac{\abs{\lambda_1}}{\lambda_0}}^n}}\;.
\end{equation} 
This shows that the distribution of the exit location is asymptotically equal to a periodically modulated
 exponential distribution. Note that the integral appearing in~\eqref{resrp6} is proportional to the expectation of $\ph_\tau$ when starting in the
quasistationary distribution. 

In order to combine the ideas based on Markov chains and on large deviations,
we will rely on the approach first used in~\cite{BG7}, and decompose the dynamics into two
subchains, the first one representing the dynamics away from the unstable orbit, and the second one representing the dynamics near the unstable orbit. We consider:
\begin{enum}
\item	A chain for the process killed upon reaching, at time $\tau_-$, a level
$1-\delta$ below the unstable periodic orbit. We denote its kernel $\Ks$. By
Assumption~\ref{assump_heteroclinic}, the first-hitting location $\ph_{\tau_-}$
will be concentrated near places $s^*+n$ where a translate
$\gamma_{\infty}(\cdot+n)$ of the minimiser
$\gamma_\infty$ crosses the level $1-\delta$. We will establish a spectral-gap
estimate for $\Ks$ (see Theorem~\ref{thm_spectralgap}), showing that
$\ph_{\tau_-}$ indeed follows a periodically modulated exponential of the form 
\begin{equation}
 \label{resrp6a}
 \bigprobin{0}{\ph_{\tau_-}\in[\ph_1,\ph_1+\Delta]}
 \simeq (\lOs)^{\ph_1} \e^{-J(\ph_1)/\sigma^2}\;,
\end{equation} 
where $J$ is periodic and minimal in points of the form $s^*+n$. 

\item 	A chain for the process killed upon reaching either the unstable
periodic orbit at $r=1$, or a level $1-2\delta$, with kernel $\Ku$. We show in
Theorem~\ref{thm_lOu} that its principal eigenvalue is of the form 
\begin{equation}
 \label{resrp7}
 \lOu = \e^{-2\lambda_+T_+}(1+\Order{\delta})\;.
\end{equation} 
Together with a large-deviation estimate, this yields a rather precise
description of the distribution of $\ph_\tau$, given the value of
$\ph_{\tau_-}$, of the form 
\begin{equation}
 \label{resrp7a} 
 \bigprobin{\ph_{\tau_-}}{\ph_{\tau}\in[\ph,\ph+\Delta]}
 \simeq \e^{-2\lambda_+T_+(\ph-\ph_{\tau_-})} 
 \exp\Bigset{-\frac{1}{\sigma^2} \Bigbrak{I_\infty +
\Order{\e^{-2\lambda_+T_+(\ph-\ph_{\tau_-})}}}}\;,
\end{equation} 
where $I_\infty$ is again related to the rate function, and the term 
$\Order{\e^{-2\lambda_+T_+(\ph-\ph_{\tau_-})}}$ can be computed explicitly to
leading order. The double-exponential dependence of~\eqref{resrp7a} on 
$2\lambda_+T_+(\ph-\ph_{\tau_-})$ is in fact what characterises the Gumbel
distribution. 
\end{enum}

By combining the two above steps, we obtain that the first-exit distribution is
given by a sum of shifted Gumbel distributions, in which each term is
associated with a translate of the optimal path $\gamma_\infty$. 


\subsection{Main result: Cycling}
\label{ssec_resmr}

In order to formulate the main result, we introduce the notation 
\begin{equation}
 \label{resmr01}
 \hper(\ph) = \frac{\e^{2\lambda_+T_+\ph}}{1-\e^{-2\lambda_+T_+}}
 \int_\ph^{\ph_+1} \e^{-2\lambda_+T_+ u} D_{rr}(1,u)\6u
\end{equation} 
for the periodic solution of the equation 
\begin{equation}
 \label{equation_for_hper} 
 \dtot{h}{\ph} = 2\lambda_+T_+ h -
D_{rr}(1,\ph)\;,
 \end{equation} 
where
\begin{equation}
 \label{resmr02}
 D_{rr}(1,\ph) = g_r(1,\ph) \transpose{g_r(1,\ph)}
\end{equation} 
measures the strength of diffusion in the direction orthogonal to the periodic
orbit. Recall that $\lambda_+T_+$ measures the growth rate per period near the
unstable periodic orbit, which is independent of the coordinate system. The
periodic function 
\begin{equation}
 \label{resmr02A}
 \theta'(\ph) = \frac{D_{rr}(1,\ph)}{2\hper(\ph)}
\end{equation} 
will provide a natural parametrisation of the orbit, in the following sense. 
Consider the 
linear approximation of the equation near the unstable orbit (assuming
$T_+=1$ for simplicity) given by 
\begin{align}
\nonumber
\6r_t &= \lambda_+ (r_t-1) \6t + \sigma g_r(1,\ph_t) \6W_t\;, \\
\6\ph_t &= \6t\;.
\label{resms10A} 
\end{align}
Then the affine change of variables $r-1=\sqrt{2\lambda_+\hper(\ph)}\,y$,
followed by the time change $s = (\theta'(\ph_t)/\lambda_+)t$
transforms~\eqref{resms10A} into 
\begin{align}
\nonumber
\6y_s &= \lambda_+ y_s \6s + \sigma \tilde g(\psi_s) \6W_s\;,
\quad \text{with }
\tilde g(\psi_s) = \frac{g_r(1,\ph_t)}{\sqrt{D_{rr}(\ph_t)}}\\
\6\psi_s &= \6s\;,
\label{resms10B} 
\end{align}
where we set $\psi=\lambda_+\theta(\ph)$. 
The new diffusion coefficient satisfies $\widetilde D_{rr}(\psi) = \tilde
g(\psi) \transpose{\tilde g(\psi)}=1$, and thus $\hpertilde(\psi)=1/2\lambda_+$
is constant. In particular if $W_t$ were one-dimensional we would have $\tilde
g(\psi)=1$. In other words, any primitive $\theta(\ph)$ of $\theta'(\ph)$ can be
thought of as a parametrisation of the unstable orbit in which the effective
transversal noise intensity is constant. 

\begin{figure}
\centerline{\includegraphics*[clip=true,height=60mm]{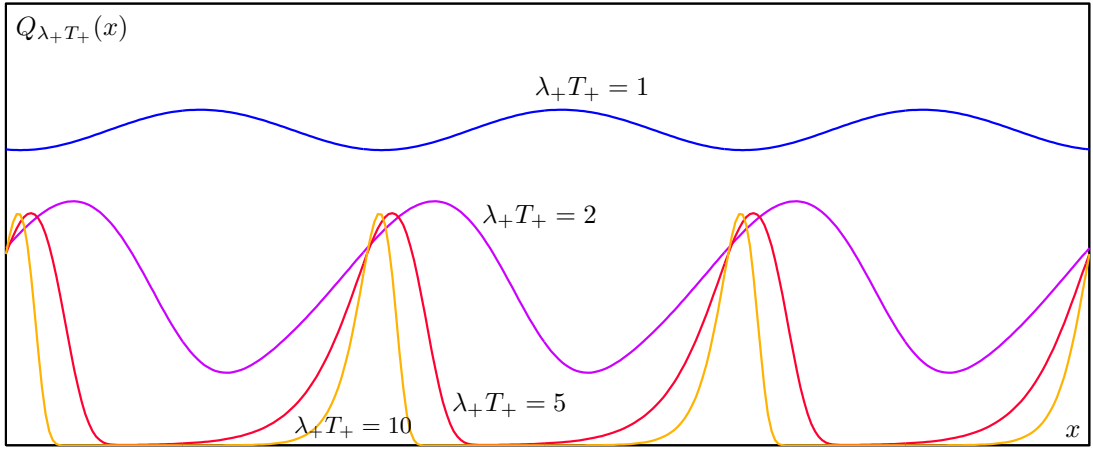}}
 \figtext{
	\writefig	14.25	0.65	$x$
	\writefig	0.3	6.0	$Q_{\lambda_+T_+}(x)$
	\writefig	7.2	5.2	$\lambda_+T_+=1$
	\writefig	6.5	3.5	$\lambda_+T_+=2$
	\writefig	6.1	1.0	$\lambda_+T_+=5$
	\writefig	4.0	0.7	{\footnotesize $\lambda_+T_+=10$}
 }
 \caption[]
 {The cycling profile $x\mapsto Q_{\lambda_+T_+}(x)$ for different values of the
parameter $\lambda_+T_+$, shown for $x\in[0,3]$.}
\label{fig_cycling}
\end{figure}

\begin{theorem}[Main result]
\label{main_theorem} 
There exist $\beta, c>0$ such that 
for any sufficiently small $\delta, \Delta>0$, there exists $\sigma_0>0$
such that the following holds: For any $r_0$ sufficiently close to $-1$ and
$\sigma< \sigma_{0}$, 
\begin{align}
\nonumber
\biggprobin{r_0,0}{\frac{\theta(\ph_\tau)}{\lambda_+T_+} \in [t,t+\Delta]}
={}& \Delta 
 C_0(\sigma) (\lambda_0)^t
 Q_{\lambda_+T_+}\biggpar{\frac{\abs{\log\sigma}
}{\lambda_+ T_+} - t + \Order{\delta}} \\
 &{}\times \biggbrak{1 + \BigOrder{\e^{-c\ph/\abs{\log\sigma}}} + 
 \Order{\delta\abs{\log\delta}} + \Order{\Delta^\beta}}\;,
 \label{resmr03}
\end{align} 
where we use the following notations:
\begin{itemiz}
\item 	$Q_{\lambda_+T_+}(x)$ is periodic with period $1$ and given be the
periodicised Gumbel
distribution 
\begin{equation}
 \label{resmr04a} 
 Q_{\lambda_+T_+}(x) = \sum_{n=-\infty}^\infty A\bigpar{\lambda_+T_+(n-x)}\;, 
\end{equation} 
where
\begin{equation}
 \label{resmr04aa} 
 A(x) = \exp\Bigset{-2x - \frac12 \e^{-2x}}
\end{equation} 
is the density of a type-$1$ Gumbel distribution with mode $-\log 2/2$ and
scale parameter $1/2$ (and thus variance $\pi^2/24$). 

\item 	$\theta(\ph)$ is the particular primitive\footnote{The differential
equation~\eqref{equation_for_hper} defining $\hper$
implies that indeed $\theta'(\ph)=D_{rr}(\ph)/(2\hper(\ph))$.}
of $\theta'(\ph)$ given by 
\begin{equation}
 \label{resmr04b}
 \theta(\ph) = \lambda_+ T_+ \ph 
 -\frac12 \log \biggbrak{\frac12 \delta^2 
\frac{\hper(\ph)}{\hper(s^*)^2} }\;,
\end{equation} 
where $s^*$ denotes the value of $\ph$ where the optimal path $\gamma_\infty$
crosses the level $1-\delta$. 
It satisfies $\theta(\ph+1)=\theta(\ph) + \lambda_+T_+$. 

\item 	$\lambda_0$ is the principal eigenvalue of the Markov chain, and
satisfies 
\begin{equation}
 \label{resmr04c}
 \lambda_0 = 1 - \e^{-H/\sigma^2}
\end{equation} 
where $H>0$ is close to the value of the rate function $I(\gamma_\infty)$.

\item 	The normalising constant $C_0(\sigma)$ is of order $\e^{-H/\sigma^2}$. 
\end{itemiz}
\end{theorem}

The proof is given in Section~\ref{sec_el}. We now comment the different terms
in the expression~\eqref{resmr03} in more detail. 

\begin{itemiz}
\item 	{\bf Cycling profile:} The function $Q_{\lambda_+T_+}$ is the announced
universal cycling profile. Relation~\eqref{resmr03} shows that the profile
is translated along the unstable orbit proportionally to $\abs{\log\sigma}$.
The intuition is that this is the time needed for the optimal path
$\gamma_\infty$ to reach a $\sigma$-neighbourhood of the unstable orbit where 
escape becomes likely. For small values of $\lambda_+T_+$, the cycling profile
is rather flat, while it becomes more and more sharply peaked as $\lambda_+T_+$
increases (\figref{fig_cycling}). 

\item 	{\bf Principal eigenvalue:} The principal eigenvalue $\lambda_0$
determines the slow exponential decay of the first-exit distribution. Writing 
$(\lambda_0)^t=\e^{-t\abs{\log\lambda_0}}$, we see that the expected
first-exit location is of order $1/\abs{\log\lambda_0}\simeq\e^{H/\sigma^2}$. 
This \lq\lq time\rq\rq\ plays the same r\^ole as Kramers' time for gradient
systems (see~\cite{Eyring,Kramers} and e.g.~\cite{Berglund_Kramers} for a recent
review of Kramers' law). 
One may obtain sharper bounds on $\lambda_0$ using, for instance, the
Donsker--Varadhan inequality~\cite{DonskerVaradhan76}.

\item 	{\bf Normalisation:} The prefactor $C_0(\sigma)$ can be estimated using
the fact that the first-exit distribution is normalised to $1$. It is of the
order $\abs{\log\lambda_0}\simeq\e^{-H/\sigma^2}$. 

\item 	{\bf Transient behaviour:} The error term
$\Order{\e^{-ct/\abs{\log\sigma}}}$ describes the transient behaviour when not
starting in the quasistationary distribution. If the initial condition is
concentrated near the stable periodic orbit, we expect the first-exit
distribution to be bounded above by the leading term in~\eqref{resmr03} during
the transient phase. 

\item 	{\bf Dependence on a level $\delta$:} While the left-hand side of~\eqref{resmr03} does not depend on $\delta$ and one would like to take the limit $\delta\to0$ on the right-hand side, this would require also to pass to the limit $\sigma\to0$ since the maximal value $\sigma_0$ depends on $\delta$ (as it does depend on $\Delta$).
\end{itemiz}

To illustrate the dependence of the first-passage distribution on the
parameters, we provide two animations, available at \\
{\tt http://www.univ-orleans.fr/mapmo/membres/berglund/simcycling.html}.\\
They show how the distribution changes with noise intensity $\sigma$ (cycling)
and orbit period $T_+$, respectively. In order to show the dependence more
clearly, the chosen parameter ranges exceed in part the domain in which our
results are applicable. 


\subsection{Discussion}
\label{ssec_resdisc}

We now present some consequences of Theorem~\ref{main_theorem} which help to
understand the result. First of all, we may consider the wrapped distribution 
\begin{equation}
 \label{resd01}
 \cW_\Delta(t) = \sum_{n=0}^\infty 
 \biggprobin{r_0,0}{\frac{\theta(\ph_\tau)}{\lambda_+T_+} \in
[n+t,n+t+\Delta]}\;,
\end{equation} 
which describes the first-hitting location of the periodic orbit without
keeping track of the winding number. Then an immediate consequence of
Theorem~\ref{main_theorem} is the following.

\begin{cor}
\label{cor_Wrap}
Under the assumptions of the theorem, we have 
\begin{equation}
 \label{resd02}
 \cW_\Delta(t) = \Delta Q_{\lambda_+T_+}\biggpar{\frac{\abs{\log\sigma}
}{\lambda_+ T_+} - t + \Order{\delta}}
\biggbrak{1 + \Order{\delta\abs{\log\delta}} + \Order{\Delta^\beta}}\;.
\end{equation} 
As a consequence, the following limit result holds:
\begin{equation}
 \label{resd03}
 \lim_{\delta,\Delta\to 0} \lim_{\sigma\to 0}
 \frac{1}{\Delta} \cW_\Delta 
 \biggpar{t + \frac{\abs{\log\sigma}}{\lambda_+T_+}}
 = Q_{\lambda_+T_+}(-t)\;.
\end{equation} 
\end{cor}

This asymptotic result stresses that the cycling profile can be recovered in
the zero-noise limit, if the system of coordinates is shifted along the orbit
proportionally to $\abs{\log\sigma}$. 
One could write similar results for the unwrapped first-hitting distribution,
but the transient term $\e^{-ct/\abs{\log\sigma}}$ would require to introduce
an additional shift of the observation window. A simpler statement can be made
when starting in the quasistationary distribution $\pi_0$, namely 
\begin{align}
\nonumber
\biggprobin{\pi_0}{\frac{\theta(\ph_\tau)}{\lambda_+T_+} \in [t,t+\Delta]}
={}& \Delta 
 C_0(\sigma) (\lambda_0)^t
 Q_{\lambda_+T_+}\biggpar{\frac{\abs{\log\sigma}
}{\lambda_+ T_+} - t + \Order{\delta}} \\
 &{}\times \biggbrak{1 + 
 \Order{\delta\abs{\log\delta}} + \Order{\Delta^\beta}}\;,
 \label{resd04}
\end{align} 
and thus
\begin{equation}
 \label{resd05}
 \lim_{\delta,\Delta\to 0} \lim_{\sigma\to 0} \frac{1}{C_0(\sigma)
(\lambda_0)^t}\frac1{\Delta} 
 \biggprobin{\pi_0}{\frac{\theta(\ph_\tau)+\abs{\log\sigma}}{\lambda_+T_+} \in
[t,t+\Delta]}
= Q_{\lambda_+T_+}(-t)\;.
\end{equation} 

We conclude with some remarks on applications and possible improvements and
extensions of Theorem~\ref{main_theorem}. 

\begin{itemiz}
\item 	{\bf Spectral decomposition:} In the proof presented here, we rely
partly on large-deviation estimates, and partly on spectral properties of
random Poincar\'e maps. By obtaining more precise information on the
eigenfunctions and eigenvalues of the Markov chain $\Ku$, one might be able to obtain the
same result without using large deviations. This is the case for the linearised
system (see Proposition~\ref{prop_lu_linear}), for which one can check that the
right eigenfunctions are similar to those of the quantum harmonic oscillator
(Gaussians multiplied by Hermite polynomials). 

\item 	{\bf Residence-time distribution:} Consider the situation where there
is a stable periodic orbit surrounding the unstable one. Then sample paths of
the system switch back and forth between the two stable orbits, in a way
strongly influenced by noise intensity and period of the orbits. The
residence-time distribution near each orbit is related to the above first-exit
distribution~\cite{BG9}, and has applications in the quantification of the
phenomenon of stochastic resonance (see also~\cite[Chapter~4]{BGbook}). 

\item 	{\bf More general geometries:} In a similar spirit, one may ask what
happens if the stable periodic orbit is replaced by a stable equilibrium point,
or some other attractor. We expect the result to be similar in such a
situation, because the presence of the periodic orbit is only felt inasmuch
hitting points of the level $1-\delta$ are concentrated within each period. 

\item 	{\bf Origin of the Gumbel distribution:} The proof shows that the
double-exponential behaviour of the cycling profile results from a combination
of the exponential convergence of the large-deviation rate function to its
asymptotic value and the exponential decay of the QSD near the unstable orbit. 
Still, it would be nice to understand whether there is a link between this exit
problem and extreme-value theory. As mentioned in the introduction, the
authors of~\cite{CerouGuyaderLelievreMalrieu12} obtained that the length of
reactive paths is also governed by a Gumbel distribution, but their proof
relies on Doob's h-transform and the exact solution of the resulting ODE, and
thus does not provide immediate insight into possible connections with
extreme-value theory. 
\end{itemiz}


\section{Coordinate systems}
\label{sec_cs}


\subsection{Deterministic system}
\label{ssec_csd}

We start by constructing polar-like coordinates for the deterministic
ODE~\eqref{csd1}. 

\begin{prop}
\label{prop_csd}
There is an open subset $\cD_1 = (-L,L)\times\fS^1$ of the cylinder, with
$L>1$, and a $\cC^2$-diffeomorphism $h: \cD_1\to\cD_0$ such that~\eqref{csd1}
is equivalent, by the transformation $z=h(r,\ph)$, to the system
\begin{equation}
\label{csd7}
\begin{split}
\dot r &= f_r(r,\ph)\;, \\
\dot \ph &= f_\ph(r,\ph)\;,
\end{split}
\end{equation}
where $f_r, f_\ph: \cD_1\to\R$ satisfy 
\begin{align}
\nonumber
f_r(r,\ph) &= \lambda_+(r-1) + \Order{(r-1)^2} \;,\\
f_\ph(r,\ph) &= \frac1{T_+} + \Order{(r-1)^2}
\label{csd8}
\end{align}
as $r\to1$, and 
\begin{align}
\nonumber
f_r(r,\ph) &= -\lambda_-(r+1) + \Order{(r+1)^2} \;,\\
f_\ph(r,\ph) &= \frac1{T_-} + \Order{(r+1)^2}
\label{csd9}
\end{align}
as $r\to-1$. Furthermore, $f_\ph(r,\ph)$ is positive, bounded away
from $0$, while $f_r(r,\ph)$ is negative for $\abs{r}<1$ and positive
for $\abs{r}>1$. 
\end{prop}
\begin{proof}
The construction of $h$ proceeds in several steps. We start by defining $h$
in a neighbourhood of $r=1$, before extending it to all of $\cD_1$. 
\begin{enum}
\item	We set $h(1,\ph) = \Gamma_+(\ph)$. 
Hence $f(\Gamma_+(\ph))=\sdpar hr(1,\ph)\dot r + 
\Gamma'_+(\ph)\dot\ph$, so that $\dot\ph=1/T_+$ and $\dot
r=0$ whenever $r=1$. 

\item	Let $u_+(0)$ be an eigenvector of the monodromy matrix $U_+(1,0)$
with eigenvalue $\e^{\lambda_+T_+}$. Then it is easy to check that 
\begin{equation}
\label{csd10:1}
u_+(\ph) = \e^{-\lambda_+T_+\ph} U_+(\ph,0) u_+(0)
\end{equation}
is an eigenvector of the monodromy matrix $U_+(\ph+1,\ph)$
with same eigenvalue $\e^{\lambda_+T_+}$, and that 
\begin{equation}
\label{csd10:2}
\dtot{}{\ph} u_+(\ph) = T_+ \bigbrak{A_+(\ph)u_+(\ph) -
\lambda_+ u_+(\ph)}\;.
\end{equation}
We now impose that 
\begin{equation}
\label{csd10:3}
h(r,\ph) = \Gamma_+(\ph) + (r-1) u_+(\ph) + \Order{(r-1)^2}
\end{equation}
as $r\to1$. This implies that 
\begin{equation}
\label{csd10:4}
f(h(r,\ph)) = f(\Gamma_+(\ph)) + (r-1) A_+(\ph)u_+(\ph) +
\Order{(r-1)^2}\;,
\end{equation}
which must be equal to 
\begin{align}
\label{csd10:5}
\dot z 
&= \bigbrak{\Gamma'_+(\ph) + u'_+(\ph)(r-1)
+ \Order{(r-1)^2}}\dot\ph + \bigbrak{u_+(\ph)  + \Order{(r-1)}}\dot{r}
\nonumber\\
&= T_+ \bigbrak{f(\Gamma_+(\ph)) + \bigpar{A_+(\ph)u_+(\ph) -
\lambda_+ u_+(\ph)}(r-1) + \Order{(r-1)^2}}\dot\ph  
\nonumber\\
&\phantom{{}=}{}{}+
\bigbrak{u_+(\ph)  + \Order{(r-1)}}\dot{r}\;.
\end{align}
Comparing with~\eqref{csd10:4} and, in a first step, projecting on a vector normal to $u_+(\ph)$ shows that
$\dot\ph = 1/T_+ + \Order{(r-1)^2} + \Order{\dot r(r-1)}$. Then, in a second step, projecting on a vector perpendicular to
$f(\Gamma_+(\ph))$ shows that $\dot r = \lambda_+(r-1) +
\Order{(r-1)^2}$, which also implies $\dot\ph = 1/T_+ + \Order{(r-1)^2}$.

\item	In order to extend $h(r,\ph)$ to all of $\cD_1$, we start by
constructing a curve segment $\Delta_0$, connecting $\Gamma_+(0)$ to some
point $\Gamma_-(\ph^\star)$ on the stable orbit, which is crossed by all
orbits of the vector field in the same direction (see~\figref{fig_csd}).
Reparametrising $\Gamma_-$
if necessary, we may assume that $\ph^\star=0$. The curve $\Delta_0$ can
be chosen to be tangent to $u_+(0)$ in $\Gamma_+(0)$, and to the similarly
defined vector $u_-(0)$ in $\Gamma_-(0)$. We set 
\begin{equation}
\label{csd10:6}
h(r,\ph) = \Gamma_-(\ph) + (r+1) u_-(\ph) + \Order{(r+1)^2}
\end{equation}
as $r\to-1$, which implies in particular the relations~\eqref{csd9}. 

The curve segment $\Delta_0$ can be parametrised by a function  $r\mapsto
h(r,0)$ which is compatible with~\eqref{csd10:4} and~\eqref{csd10:6}, that
is, $\sdpar hr(\pm1,0) = u_\pm(0)$. We proceed similarly with each element
of a smooth deformation $\set{\Delta_\ph}_{\ph\in\fS^1}$ of
$\Delta_0$, where $\Delta_\ph$ connects $\Gamma_+(\ph)$ to
$\Gamma_-(\ph)$ and is tangent to $u_\pm(\ph)$. The parametrisation
$r\mapsto h(r,\ph)$ of $\Delta_\ph$ can be chosen in such a way that
whenever $\ph<\ph'$, the orbit starting in $h(r,\ph)\in\cS$ first
hits $\Delta_{\ph'}$ at a point $h(r',\ph')$ with $r'<r$. This
guarantees that 
\begin{equation}
\label{csd10:7}
\begin{split}
\dot r &= f_r(r,\ph)\;, \\
\dot \ph &= f_\ph(r,\ph)\;
\end{split}
\end{equation}
with $f_r(r,\ph)<0$ for $(r,\ph)\in\cS$. 

\item	We can always assume that $f_\ph(r,\ph)>0$, replacing, if
necessary, $\ph$ by $\ph+\delta(r)$ for some function $\delta$
vanishing in $r=1$. 
\qed
\end{enum}
\renewcommand{\qed}{}
\end{proof}

\begin{remark}
\label{rem_csd}
One can always use $\ph$ as new time variable, and rewrite~\eqref{csd7}
as the one-dimensional, non-autonomous equations 
\begin{equation}
\label{csd11}
\dtot r\ph = \frac{f_r(r,\ph)}{f_\ph(r,\ph)} \bydef
F(r,\ph)\;.
\end{equation}
Note, in particular, that 
\begin{equation}
\label{csd12}
F(\pm1,\ph) = T_\pm \lambda_\pm (r\mp1)) + \bigOrder{(r\mp1)^2}\;.
\end{equation}
\end{remark}


\subsection{Stochastic system}
\label{ssec_css}

We now turn to the SDE~\eqref{css1} which is equivalent, via the transformation
$z=h(r,\ph)$
of Proposition~\ref{prop_csd}, to a system of the form 
\begin{equation}
\label{css2}
\begin{split}
\6r_t &= f_r(r_t,\ph_t;\sigma) \,\6t + \sigma g_r(r_t,\ph_t) \6W_t\;,\\
\6\ph_t &= f_\ph(r_t,\ph_t;\sigma) \,\6t 
+ \sigma g_\ph(r_t,\ph_t) \6W_t\;.
\end{split}
\end{equation}
In fact, It\^o's formula shows that $f=\transpose{(f_r,f_\ph)}$ and
$g=\transpose{(g_r,g_\ph)}$, where $g_r$ and $g_\ph$ are $(1\times k)$-matrices,
satisfying 
\begin{align}
\nonumber
g(h(r,\ph)) ={}& \sdpar hr(r,\ph) g_r(r,\ph) + \sdpar
h\ph(r,\ph) g_\ph(r,\ph) \\
\label{css3}
f(h(r,\ph)) ={}& \sdpar hr(r,\ph) f_r(r,\ph) + \sdpar
h\ph(r,\ph) f_\ph(r,\ph) \\
\nonumber
{}&+ \frac12 \sigma^2 \bigbrak{\sdpar h{rr}(r,\ph)
g_r\transpose{g_r}(r,\ph) + 2\sdpar h{r\ph}(r,\ph)
g_r\transpose{g_\ph}(r,\ph) + \sdpar h{\ph\ph}(r,\ph)
g_\ph\transpose{g_\ph}(r,\ph)}\;.
\end{align}
The first equation allows to determine $g_r$ and $g_\ph$, by projection
on $\sdpar hr$ and $\sdpar h\ph$. The second one shows that 
\begin{align}
\nonumber
f_r(r,\ph;\sigma) &= f_r^0(r,\ph) + \sigma^2 f_r^1(r,\ph)\;,\\
f_\ph(r,\ph;\sigma) &= f_\ph^0(r,\ph) + 
\sigma^2 f_\ph^1(r,\ph)\;,
\label{css4}
\end{align}
where $f_r^0$ and $f_\ph^0$ are the functions of
Proposition~\ref{prop_csd}. 

A drawback of the system~\eqref{css2} is that the drift term $f_r$ in
general no longer vanishes in $r=\pm1$. This can be seen as an effect
induced by the curvature of the orbit, since $f_r^1(\pm1,\ph)$ depends
on $\Gamma'_\pm(\ph)$ and $u'_\pm(\ph)$. This problem can, however,
be solved by a further change of variables. 

\begin{prop}
\label{prop_css}
There exists a
change of variables of the form $y=r+\Order{\sigma^2}$, leaving $\ph$ unchanged,  such that the drift
term for $\6y_t$ vanishes in $y=\pm1$. 
\end{prop}
\begin{proof}
We shall look for a change of variables of the form 
\begin{equation}
\label{css5:1}
y = Y(r,\ph) = r - \sigma^2 \bigbrak{\Delta_-(\ph)(r-1) +
\Delta_+(\ph)(r+1)}\;,
\end{equation}
where $\Delta_\pm(\ph)$ are periodic functions, representing the shift of
variables near the two periodic orbits. Note that $y=1$ for 
\begin{equation}
\label{css5:2}
r = 1 + 2\sigma^2 \Delta_+(\ph) + \Order{\sigma^4}\;.
\end{equation}
Using It\^o's formula, one obtains a drift term for $\6y_t$ satisfying 
\begin{align}
\nonumber
f_y(1,\ph) 
={}& f_r(1 + 2\sigma^2 \Delta_+(\ph) + \Order{\sigma^4}, \ph) \bigbrak{1-\sigma^2\brak{\Delta_-(\ph) + \Delta_+(\ph)}}\\
\label{css5:3}
{}&- \sigma^2 \bigbrak{2 \Delta'_+(\ph)f_\ph(1,\ph) +
\Order{\sigma^2}} \\
\nonumber
{}&- \frac12 \sigma^4
\bigbrak{
4\Delta'_+(\ph)
g_r(1,\ph) \transpose{g_\ph(1,\ph)} 
+2\Delta''_+(\ph)g_\ph(1,\ph) \transpose{g_\ph(1,\ph)} 
+ \Order{\sigma^2}}\;,
\end{align}
where the terms $\Order{\sigma^2}$ depend on $\Delta_\pm$ and $\Delta'_\pm$.
Using~\eqref{csd8}, we see that, in order that $f_y(1,\ph)$ vanishes,
$\Delta_+(\ph)$ has to satisfy an equation of the form 
\begin{equation}
\label{css5:4}
\lambda_+ \Delta_+(\ph) - \frac1{T_+}\Delta_+'(\ph) 
+ r(\ph,\Delta_\pm(\ph),\Delta'_\pm(\ph)) - \sigma^2
b(\ph,\Delta_\pm(\ph))\Delta_+''(\ph) = 0\;,
\end{equation}
where $r(\ph,\Delta_\pm,\Delta'_\pm) = f^1_r(1,\ph) +
\Order{\sigma^2}$. Note that $b(\ph,\Delta)>0$ is bounded away from zero
for small $\sigma$ by our ellipticity assumption on $g$. 
A similar equation is obtained for $\Delta_-(\ph)$. If
$\Delta=(\Delta_+,\Delta_-)$ and $\Xi = (\Delta'_+,\Delta'_-)$, we arrive at
a system of the form
\begin{equation}
\label{css5:5}
\begin{split}
\sigma^2 B(\ph,\Delta) \Xi' &= - D\Xi + \Lambda\Delta +
R(\ph,\Delta,\Xi) \;,\\
\Delta' &= \Xi \;,\\
\ph' &= 1\;.
\end{split}
\end{equation} 
Here $D$ denotes a diagonal matrix with entries $1/T_+$ and $1/T_-$, and
$\Lambda$ denotes a diagonal matrix with entries $\lambda_\pm$. The
system~\eqref{css5:5} is a slow--fast ODE, in which $\Xi$ plays the r\^ole
of the fast variable, and $(\Delta,\ph)$ are the slow variables. The fast
vector field vanishes on a normally hyperbolic slow manifold of the form
$\Xi = \Xi^*(\Delta,\ph)$, where 
\begin{equation}
\label{css5:6}
\Xi_\pm^*(\Delta,\ph) = T_\pm \bigbrak{\lambda_\pm \Delta_\pm(\ph) +
f^1_r(\pm1,\ph)} + \Order{\sigma^2}\;.
\end{equation}
By Fenichel's theorem~\cite{Fenichel}, there exists an invariant manifold
$\Xi=\overline\Xi(\ph,\Delta)$ in a $\sigma^2$-neighbour\-hood of the slow
manifold. The reduced equation on this invariant manifold takes the form 
\begin{equation}
\label{css5:7}
\Delta'_\pm = T_\pm \bigbrak{\lambda_\pm \Delta_\pm(\ph) +
f^1_r(\pm1,\ph)} + \Order{\sigma^2}\;.
\end{equation}
The limiting equation obtained by setting $\sigma$ to zero admits an
explicit periodic solution. Using standard arguments of regular perturbation
theory, one then concludes that the full equation~\eqref{css5:7} also admits
a periodic solution. 
\end{proof}



\section{Large deviations}
\label{sec_ld}

In this section, we consider the dynamics near the unstable periodic orbit on
the level of large deviations. We want to estimate the infimum $I_\ph$ of the
rate function for the event $\Gamma(\delta)$ that a sample path, starting at
sufficiently small distance $\delta$ from the unstable orbit, reaches the 
unstable orbit at the moment when the angular variable has increased by $\ph$. 

Consider first the system linearised around the unstable orbit, given by 
\begin{equation}
 \label{ld101}
 \dtot{}{\ph} 
 \begin{pmatrix}
r \\ p_r
\end{pmatrix}
=
\begin{pmatrix}
\lambda_+T_+  & D_{rr}(0,\ph) \\
0 & -\lambda_+ T_+
\end{pmatrix}
\begin{pmatrix}
r \\ p_r
\end{pmatrix}
\;.
\end{equation} 
(We have redefined $r$ so that the unstable orbit is in $r=0$.)
Its solution can be written in the form 
\begin{equation}
 \label{ld102}
 \begin{pmatrix}
  r(\ph) \\ p_r(\ph)
 \end{pmatrix}
= 
\begin{pmatrix}
\e^{\lambda_+T_+(\ph-\ph_0)} 
& \e^{\lambda_+T_+(\ph-\ph_0)} 
\displaystyle\vrule height 16pt depth 16pt width 0pt
\int_{\ph_0}^\ph \e^{-2\lambda_+T_+ (u-\ph_0)} D_{rr}(0,u) \6u \\
0 
& \e^{-\lambda_+T_+(\ph-\ph_0)} 
\end{pmatrix}
\begin{pmatrix}
r(\ph_0) \\ p_r(\ph_0) 
\end{pmatrix}\;.
\end{equation} 
The off-diagonal term of the above fundamental matrix can also be expressed in
the form 
\begin{equation}
 \label{ld103}
 \e^{\lambda_+T_+(\ph-\ph_0)}\hper(\ph_0) - 
\e^{-\lambda_+T_+(\ph-\ph_0)}\hper(\ph)\;,
\end{equation} 
where 
\begin{equation}
 \label{ld104}
 \hper(\ph) = \frac{\e^{2\lambda_+T_+\ph}}{1-\e^{-2\lambda_+T_+}}
 \int_\ph^{\ph_+1} \e^{-2\lambda_+T_+ u} D_{rr}(0,u)\6u
\end{equation} 
is the periodic solution of the equation 
$\6h/\6\ph = 2\lambda_+T_+ h - D_{rr}(0,\ph)$. The expression~\eqref{ld103} shows
that for initial conditions satisfying $r(\ph_0) = -\hper(\ph_0) p_r(\ph_0)$,
the orbit $(r(\ph),p_r(\ph))$ will converge to $(0,0)$. The stable manifold of the
unstable orbit is thus given by the equation $r=-\hper(\ph)p_r$. 

We consider now the following situation: Let $(r(0),p_r(0))$ belong to the
stable manifold. The orbit starting in this point takes an infinite time to
reach the unstable orbit, and gives rise to a value $I_\infty$ of the rate
function. We want to compare this value to the rate function $I_\ph$ of an
orbit starting at the same $r(0)$, but reaching $r=0$ in finite time $\ph$. 

Recall that the rate function has the expression 
\begin{equation}
\label{ld104b}
I(\gamma) = \dfrac12 \int_0^T \transpose{\psi_s} D(\gamma_s) \psi_s
\,\6s 
= \frac12 \int \Bigbrak{D_{rr} p_r^2 + 2 D_{r\ph} p_r p_\ph + D_{\ph\ph}
p_\ph^2} \6\ph\;.
\end{equation}
However, $p_\ph$ can be expressed in terms of $r,\ph$ and $p_r$ using the
Hamiltonian, and is of order $r^2+p_r^2$. Thus the leading term in the rate
function near the unstable orbit is $D_{rr} p_r^2$. As a first approximation
we may thus consider 
\begin{equation}
 \label{ld104c}
 I^0(\gamma) = \frac12 \int D_{rr} p_r^2\6\ph\;.
\end{equation} 

\begin{prop}[Comparison of rate functions in the linear case]
Denote by $I^0_\infty$ and $I^0_\ph$ the minimal value of the rate function $I^0$ for orbits starting in $r(0)$ and reaching the unstable orbit in infinite time or in time $\ph$, respectively. We have 
\begin{equation}
 \label{ld105}
 I^0_\ph - I^0_\infty = \frac12 \delta^2 \e^{-2\lambda_+T_+\ph}
\frac{\hper(\ph)}{\hper(0)^2} 
 \Bigbrak{1 + \bigOrder{\e^{-2\lambda_+T_+\ph}}}\;.
\end{equation} 
\end{prop}
\begin{proof}
Let $(r^0,p^0_r)(u)$ be the orbit with initial condition $(r(0),p_r(0))$, and 
$(r^1,p^1_r)(u)$ the one with initial condition $(r(0),p_r(0)+q)$. Then we have by~\eqref{ld102} with $\ph_0=0$,
\eqref{ld103} and the relation $r(0)=-\hper(0)p_r(0)$  
\begin{align}
\nonumber
r^1(u) &= \e^{-\lambda_+T_+u} \frac{\hper(u)}{\hper(0)} r(0) 
+  q\Bigbrak{ \e^{\lambda_+T_+u}\hper(0) -  \e^{- \lambda_+T_+u}\hper(u)}\;, \\
p^1_r(u) &= \e^{-\lambda_+T_+u} (p_r(0)+q)\;.
\label{ld106:1} 
\end{align}
The requirement $r^1(\ph)=0$ implies 
\begin{equation}
 \label{ld106:2} 
 q = \e^{-2\lambda_+T_+\ph}
\frac{\hper(\ph)}{\hper(0)} 
p_r(0)\Bigbrak{1+\bigOrder{\e^{-2\lambda_+T_+\ph}}}
\;.
\end{equation} 
Since the solutions starting on the stable manifold satisfy $p_r^0(u) =
\e^{-2\lambda_+T_+u} p_r(0)$, we get 
\begin{equation}
 \label{ld106:3}
p_r^1(u)^2 -  p_r^0(u)^2 
= 2q p_r^0(u) \e^{-\lambda_+T_+u} + q^2 \e^{-2\lambda_+T_+u}\;.
\end{equation} 
The difference between the two rate functions is thus given by 
\begin{equation}
 \label{ld106:4}
 2(I^0_\ph - I^0_\infty) = 
 (2qp_r(0)+q^2) \int_0^\ph \e^{-2\lambda_+T_+ u}D_{rr}(0,u)\6u 
 - p_r(0)^2 \int_\ph^\infty \e^{-2\lambda_+T_+ u}D_{rr}(0,u)\6u\;.
\end{equation} 
Using the relations (cf.~\eqref{ld103}) 
\begin{align}
\nonumber
\int_0^\ph \e^{-2\lambda_+T_+ u}D_{rr}(0,u)\6u 
&= \hper(0) - \e^{-2\lambda_+T_+ \ph} \hper(\ph)\;, \\
\nonumber
\int_0^\infty \e^{-2\lambda_+T_+ u}D_{rr}(0,u)\6u 
&= \hper(0)\;, \\
\int_\ph^\infty \e^{-2\lambda_+T_+ u}D_{rr}(0,u)\6u 
&= \e^{-2\lambda_+T_+ \ph} \hper(\ph)
\label{ld106:5} 
\end{align}
and $\hper(0)p_r(0)=-r(0)=-\delta$ yields the result. 
\end{proof}

We can now draw on standard perturbation theory to obtain the following result
for the nonlinear case.

\begin{prop}[Comparison of rate functions in the nonlinear case]
\label{prop_ldp_nonlinear}
For sufficiently small $\delta$, the infimum $I_\ph$ of the
rate function for the event $\Gamma(\delta)$ satisfies
\begin{equation}
 \label{ld107}
 I_\ph - I_\infty = \frac12 \delta^2 \e^{-2\lambda_+T_+\ph}
\frac{\hper(\ph)}{\hper(0)^2} 
 \Bigbrak{1 + \bigOrder{\e^{-2\lambda_+T_+\ph}} + \Order{\delta}}\;.
\end{equation}
\end{prop}
\begin{proof}
Writing $(r,\ph,p_r,p_\ph)=\delta(\bar r,\bar \ph,\bar p_r,\bar p_\ph)$, we can
consider the nonlinear terms as a perturbation of order $\delta$. Since the
solutions we consider decay exponentially, the stable manifold theorem and 
a Gronwall argument allow to bound the effect of nonlinear terms by a
multiplicative error of the form $1+\Order{\delta}$. 
\end{proof}



\section{Continuous-space Markov chains}
\label{sec_mc}


\subsection{Eigenvalues and eigenfunctions}
\label{ssec_mce}

Let $E\subset\R$ be an interval, equipped with the Borel $\sigma$-algebra. 
Consider a Markov kernel 
\begin{equation}
 K(x,\6y) = k(x,y)\6y
\end{equation} 
with density $k$ with respect to Lebesgue measure. We assume $k$ to be
continuous and square-integrable. We allow for $K(x,E)<1$, that
is, the kernel may be substochastic. In that case, we add a cemetery state
$\partial$ to $E$, so that $K$ is stochastic on $E\cup\partial$. 
Given an initial condition $X_0$, the kernel $K$ generates a Markov chain
$(X_0,X_1,\dots)$ via 
\begin{equation}
 \prob{X_{n+1}\in A} = \int_E \prob{X_n\in\6x} K(x,A)\;.
\end{equation} 
We write the natural action of the kernel on bounded measurable functions $f$
as
\begin{equation}
\label{K2} 
 (Kf)(x) \defby \bigexpecin{x}{f(X_1)}
 = \int_E k(x,y)f(y)\6y\;.
\end{equation} 
For a finite signed measure $\mu$ with density $m$, we set 
\begin{equation}
 (\mu K)(\6y) \defby \expecin{\mu}{X_1\in\6y}=
(mK)(y)\6y\;,
\end{equation} 
where 
\begin{equation} 
 (mK)(y) = \int_E m(x)\6x\, k(x,y)\;.
\end{equation} 
We know by the work of Fredholm \cite{Fredholm_1903} that the integral equation 
\begin{equation}
\label{int01} 
 (Kf)(x) - \lambda f(x)= g(x)
\end{equation} 
can be solved for any $g$, if and only if 
$\lambda$ is not an eigenvalue, i.e., the eigenvalue equation 
\begin{equation}
 (Kh)(x) = \lambda h(x)
\end{equation} 
admits no nontrivial solution. All eigenvalues $\lambda$ have finite
multiplicity, and the properly normalised left and right eigenfunctions $h_n^*$
and $h_n$ form a complete orthonormal basis, that is, 
\begin{equation}
\int_E h_n^*(x)h_m(x)\6x = \delta_{nm} 
\qquad\text{and}\qquad 
 \sum_n h_n^*(x)h_n(y) = \delta(x-y)\;.
\end{equation}
Jentzsch~\cite{Jentzsch1912} proved that if $k$ is positive, there exists a
simple eigenvalue $\lambda_0\in(0,\infty)$, which is strictly larger in module than
all other eigenvalues. It is called the \emph{principal eigenvalue}. The
associated eigenfunctions $h_0$ and $h_0^*$ are positive.
Birkhoff~\cite{Birkhoff1957} has obtained the same result under weaker
assumptions on $k$. We call the probability measure 
$\pi_0$ given by 
\begin{equation}
 \pi_0(\6x) = \frac{h^*_0(x)\6x}{\int_E h^*_0(y)\6y}
\end{equation} 
the \emph{quasistationary distribution}\/ of the Markov chain. It describes the
asymptotic distribution of the process conditioned on having survived. 

Given a Borel set $A\subset E$, we introduce the stopping times
\begin{align}
\nonumber
 \tau_A &= \tau_A(x) = \inf\setsuch{t\geqs1}{X_t\in A}\;, \\
 \sigma_A &= \sigma_A(x) = \inf\setsuch{t\geqs0}{X_t\in A}\;,
\end{align}
where the optional argument $x$ denotes the initial condition. 
Observe that $\tau_A(x)=\sigma_A(x)$ if $x\in\Ac$ while
$\sigma_A(x)=0<1\leqs\tau_A(x)$ if $x\in A$. The stopping times $\tau_A$ and
$\sigma_A$ may be infinite because the Markov chain can reach the cemetery
state before hitting $A$ (and, for the moment,  we also don't assume that the
chain conditioned to survive is recurrent). 

Given $u\in\C$, we define the Laplace transforms 
\begin{align}
\nonumber
G^u_A(x) &= \bigexpecin{x}{\e^{u\tau_A}\indexfct{\tau_A<\infty}}\;, \\
H^u_A(x) &= \bigexpecin{x}{\e^{u\sigma_A}\indexfct{\sigma_A<\infty}}\;.
\end{align}
Note that $G^u_A=H^u_A$ in $\Ac$ while $H^u_A=1$ in $A$. The following result is
easy to check by splitting the expectation defining $G^u_A$ according to the
location of $X_{1}$:

\begin{lemma}
\label{lem_mc0}
Let 
\begin{equation}
\label{mc5A}
\gamma(A)  = \sup_{x\in\Ac} K(x,\Ac)
= \sup_{x\in\Ac} \bigprobin{x}{X_1\in\Ac}
\;.
\end{equation}
Then $G^u_A(x)$ is analytic in $u$ for $\re u < \log\gamma(A)^{-1}$, i.e., 
for $\abs{\e^{-u}} > \gamma(A)$, and for these $u$ it satisfies the bound  
\begin{equation}
\label{mc5B}
\sup_{x\in\Ac} \bigabs{G^u_A(x)} \leqs
\frac1{\abs{\e^{-u}}-\gamma(A)}\;.
\end{equation}
\end{lemma}

The main interest of the Laplace transforms lies in the following result, which
shows that $H^u_A$ is \lq\lq almost an eigenfunction\rq\rq, if $G^u_A$ varies
little in $A$. 

\begin{lemma}
\label{lem_mc1}
For any $u\in\C$ such that $G^u_A$ and $K^u_A$ exist, 
\begin{equation}
\label{mc6}
KH^u_A = \e^{-u}G^u_A\;.
\end{equation} 
\end{lemma}
\begin{proof}
Splitting according to the location of $X_1$, we get 
\begin{align}
\nonumber
(KH^u_A)(x) 
&= \Bigexpecin{x}{
\bigexpecin{X_1}{\e^{u\sigma_A} \indexfct{\sigma_A<\infty}}  
}  \\
\nonumber
&= 
\Bigexpecin{x}{\indexfct{X_1\in A}
\bigexpecin{X_1}{\e^{u\sigma_A} \indexfct{\sigma_A<\infty}}  
} +
\Bigexpecin{x}{\indexfct{X_1\in \Ac}
\bigexpecin{X_1}{\e^{u\sigma_A} \indexfct{\sigma_A<\infty}}  
} \\
\nonumber
&= 
\bigexpecin{x}{\indexfct{\tau_A=1} 
} +
\bigexpecin{x}{\e^{u(\tau_A-1)} \indexfct{1<\tau_A<\infty}
} \\
&= \bigexpecin{x}{\e^{u(\tau_A-1)} \indexfct{\tau_A<\infty}}
= \e^{-u} G^u_A(x)\;.
\label{mc6:1}
\end{align}
\end{proof}

We have the following relation between an eigenfunction inside and outside $A$. 

\begin{prop}
\label{prop_mc1}
Let $h$ be an eigenfunction of $K$ with eigenvalue $\lambda=\e^{-u}$. Assume
there is a set $A\subset E$ such that 
\begin{equation}
\label{mc9}
\abs{\e^{-u}} > \gamma(A) = \sup_{x\in\Ac}
\bigprobin{x}{X_1\in\Ac}\;. 
\end{equation}
Then 
\begin{equation}
\label{mc10}
h(x) = \bigexpecin{x}{\e^{u\tau_A} h(X_{\tau_A})
\indexfct{\tau_A<\infty}}
\end{equation}
for all $x\in E$. 
\end{prop}
\begin{proof}
The eigenvalue equation can be written in the form 
\begin{equation}
\label{mc10:1}
\e^{-u}h(x) = (Kh)(x) 
= \bigexpecin{x}{h(X_1)\indexfct{X_1\in A}} 
+ \bigexpecin{x}{h(X_1)\indexfct{X_1\in \Ac}}\;.
\end{equation}
Consider first the case $x\in\Ac$. Define a linear operator
$\cT$ on the Banach space $\cX$ of continuous functions $f:\Ac\to\C$
equipped with the supremum norm, by 
\begin{equation}
\label{mc10:2}
(\cT f)(x) = \bigexpecin{x}{\e^uh(X_1)\indexfct{X_1\in A}} 
+ \bigexpecin{x}{\e^uf(X_1)\indexfct{X_1\in \Ac}}\;.
\end{equation}
Is is straightforward to check that under Condition~\eqref{mc9}, $\cT$ is a
contraction. Thus it admits a unique fixed point in $\cX$, which must
coincide with $h$. Furthermore, let $h_n$ be a sequence of functions
in $\cX$ defined by $h_0=0$ and $h_{n+1}=\cT h_n$ for all $n$. Then
one can show by induction that 
\begin{equation}
\label{mc10:3}
h_n(x) = \bigexpecin{x}{\e^{u\tau_A} h(X_{\tau_A})
\indexfct{\tau_A\leqs n}}\;.
\end{equation} 
Since $\lim_{n\to\infty}h_n(x)=h(x)$ for all $x\in\Ac$,
\eqref{mc10} holds for these $x$. 
It remains to show that~\eqref{mc10} also holds for $x\in A$. This follows by a
similar computation as in the proof of Lemma~\ref{lem_mc1}. 
\end{proof}

The following result provides a simple way to estimate the principal eigenvalue
$\lambda_0$. 

\begin{prop}
\label{prop_estimate_l0}
For any $n\geqs1$, and any interval $A\subset E$ with positive Lebesgue
measure, we have 
\begin{equation}
 \label{mc100}
 \biggbrak{\inf_{x\in A} K_{n}(x,A)}^{1/n}
 \leqs \lambda_0 \leqs 
 \biggbrak{\,\sup_{x\in E} K_{n}(x,E)}^{1/n}\;.
\end{equation} 
\end{prop}
\begin{proof}
Since the principal eigenvalue of $K_{n}$ is equal to $\lambda_0^n$, it suffices
to prove the relation for $n=1$. Let $x^*$ be the point where $h_0(x)$ reaches
its supremum. Then the eigenvalue equation yields 
\begin{equation}
 \label{mc101:1}
 \lambda_0 = \int_E k(x^*,y) \frac{h_0(y)}{h_0(x^*)}\6y
 \leqs K(x^*,E)\;,
\end{equation} 
which proves the upper bound. For the lower bound, we use 
\begin{equation}
 \label{mc101:2}
 \lambda_0 \int_A h^*_0(y)\6y 
 = \int_E h^*_0(x)K(x,A)\6x \geqs \inf_{x\in A}K(x,A) \int_A h^*_0(y)\6y\;,
\end{equation} 
and the integral over $A$ can be divided out since $A$ has positive Lebesgue
measure. 
\end{proof}

The following result allows to bound the spectral gap, between $\lambda_0$ and
the remaining eigenvalues, under slightly weaker assumptions than the uniform
positivity condition used in~\cite{Birkhoff1957}. 

\begin{prop}
\label{prop_spectralgap} 
Let $A$ be an open subset of $E$. 
Assume there exists $m:A\to(0,\infty)$ such that 
\begin{equation}
 m(y) \leqs k(x,y) \leqs L m(y)
 \qquad
 \forall x,y\in A
\end{equation} 
holds with a constant $L$ satisfying $\lambda_0L>1$. 
Then any eigenvalue $\lambda\neq\lambda_0$ of $K$ satisfies 
\begin{equation}
\label{bound_in_A} 
 \abs{\lambda} \leqs 
 \max\biggset{2\gammabar(A) \;,\;
 \lambda_0L - 1 + p_{\text{\rm kill}}(A) + \gammabar(A) 
 \frac{\lambda_0L}
 {\lambda_0L-1}
 \biggbrak{1+\frac{1}{\lambda_0 -\gammabar(A)}}
 }
\end{equation} 
where 
\begin{equation}
 \gammabar(A) = \sup_{x\in E} K(x,\Ac)
 \quad\text{and}\quad
 p_{\text{\rm kill}}(A) = \sup_{x\in A} [1-K(x,E)]
 \;.
\end{equation} 
\end{prop}

\begin{remark}
Proposition~\ref{prop_estimate_l0} shows that 
$\lambda_0\geqs 1-  p_{\text{\rm kill}}(A) - \gamma(A)$. Thus, if $A$ is chosen in such a way  that
$\gammabar(A)$ and  $ p_{\text{\rm kill}}(A) $ are small, the
bound~\eqref{bound_in_A} reads 
\begin{equation}
 \abs{\lambda} \leqs L - 1 + \bigOrder{\gammabar(A)} + \bigOrder{ p_{\text{\rm
kill}}(A) }\;.
\end{equation} 
\end{remark}

\begin{proof}
The eigenvalue equation for $\lambda$ and orthogonality of the eigenfunctions
yield 
\begin{align}
\lambda h(x) &= \int_E k(x,y) h(y) \6y \;, \\
0 &= \int_E h_0^*(y) h(y) \6y\;.
\end{align}
For any $\kappa>0$ we thus have 
\begin{equation}
 \lambda h(x) = \int_E \bigbrak{\kappa h_0^*(y) - k(x,y)} h(y)\6y\;.
\end{equation} 
Let $x_0$ be the point in $A$ where $\abs{h}$ reaches its supremum. Evaluating
the last equation in $x_0$ we obtain 
\begin{equation}
\label{two_integrals} 
 \abs{\lambda} \leqs \int_A \bigabs{\kappa h_0^*(y) - k(x_0,y)}\6y 
 + \int_{\Ac} \Bigbrak{\kappa h_0^*(y) + k(x_0,y)}
\frac{\abs{h(y)}}{\abs{h(x_0)}} \6y\;.
\end{equation} 
We start by estimating the first integral. Since for all $y\in A$, 
\begin{equation}
 \lambda_0 h^*_0(y) 
 = \int_E h^*_0(x)k(x,y)\6x \geqs 
 m(y) \int_A h^*_0(x)\6x \;,
\end{equation} 
choosing $\kappa =\lambda_0 L (\int_A h^*_0(x)\6x)^{-1}$ allows to remove the
absolute values so that 
\begin{equation}
 \int_A \bigabs{\kappa h_0^*(y) - k(x_0,y)}\6y 
 \leqs \lambda_0 L - K(x_0,A) 
 \leqs \lambda_0 L - 1 + \gammabar(A) + p_{\text{\rm kill}}(A) \;.
\end{equation} 
From now on, we assume $\abs{\lambda}\geqs2\gammabar(A)$, since otherwise there
is
nothing to show. 
In order to estimate the second integral in~\eqref{two_integrals}, we first use
Proposition~\ref{prop_mc1} with $\e^{-u}=\lambda$ and Lemma~\ref{lem_mc0} to
get for all $x\in\Ac$ 
\begin{equation}
 \abs{h(x)} \leqs
\Bigexpecin{x}{\e^{(\re u)\tau_A}\bigabs{h(X_{\tau_A})}\indexfct{\tau_A<\infty}}
\leqs \frac{\abs{h(x_0)}}{\abs{\lambda}-\gammabar(A)}\;.
\end{equation} 
The second integral is thus bounded by 
\begin{equation}
 \frac{1}{\abs{\lambda}-\gammabar(A)} 
 \int_{\Ac} \Bigbrak{\kappa h_0^*(y) + k(x_0,y)} \6y
 \leqs
 \frac{1}{\abs{\lambda}-\gammabar(A)} 
 \Biggbrak{\lambda_0L 
 \frac{\displaystyle\int_{\Ac} h_0^*(y)\6y}{\displaystyle\int_A h_0^*(y)\6y}
 + \gammabar(A)}\;.
\end{equation} 
Now the eigenvalue equation for $\lambda_0$ yields 
\begin{equation}
 \lambda_0 \int_{\Ac} h_0^*(y)\6y
 = \int_E h_0^*(x)K(x,\Ac)\6x \leqs \gammabar(A) \int_E h_0^*(x)\6x\;.
\end{equation} 
Hence the second integral can be bounded by 
\begin{equation}
 \frac{1}{\abs{\lambda}-\gammabar(A)} 
 \biggbrak{\lambda_0L \frac{\gammabar(A)}{\lambda_0-\gammabar(A)} +
\gammabar(A)}\;.
\end{equation} 
Substituting in~\eqref{two_integrals}, we thus get 
\begin{equation}
 \abs{\lambda} \leqs \lambda_0 L - 1 + \gammabar(A) + p_{\text{\rm kill}}(A)
 + \frac{\gammabar(A)}{\abs{\lambda}-\gammabar(A)} 
 \biggbrak{1 + \frac{\lambda_0L}{\lambda_0-\gammabar(A)}}\;.
\end{equation} 
Now it is easy to check the following fact: Let $\abs{\lambda}, \alpha, \beta,
\gammabar$ be positive numbers such that $\alpha, \abs{\lambda} > \gammabar$.
Then 
\begin{equation}
 \abs{\lambda} \leqs \alpha + \frac{\beta}{\abs{\lambda}-\gammabar}
 \quad
 \Rightarrow
 \quad
 \abs{\lambda} \leqs \alpha + \frac{\beta}{\alpha-\gammabar}\;.
\end{equation} 
This yields the claimed result. 
\end{proof}


\subsection{Harmonic measures}
\label{ssec_mhm}

Consider an SDE in $\R^2$ given by 
\begin{equation}
\label{SDE_hm} 
 \6x_t = f(x_t)\6t + \sigma g(x_t)\6W_t\;,
\end{equation} 
where $(W_t)_t$ is a standard $k$-dimensional Brownian motion, $k\geqs 2$, on
some probability space 
$(\Omega,\cF,\fP)$. 
We denote by 
\begin{equation}
 \cL = \sum_{i=1}^2 f_i \dpar{}{x_i} 
 + \frac{\sigma^2}{2} \sum_{i,j=1}^2
\bigpar{g\transpose{g}}_{ij}\dpar{^2}{x_i\partial
x_j} 
\end{equation}
the infinitesimal generator of the associated diffusion. 
Given a bounded open set $\cD\subset\R^2$ with Lipschitz boundary
$\partial\cD$, we are interested in properties of the first-exit
location $x_\tau\in\partial\cD$, where 
\begin{equation}
 \tau = \tau_\cD = \inf\setsuch{t>0}{x_t\not\in\cD}
\end{equation} 
is the first-exit time from $\cD$. 
We will assume that $f$ and $g$ are uniformly bounded in $\overline\cD$, and
that $g$ is uniformly elliptic in $\overline\cD$. 
Dynkin's formula and Riesz's representation theorem imply the existence of a 
\defwd{harmonic measure}\/ $H(x,\6y)$, such that  
\begin{equation}
 \bigprobin{x}{x_\tau \in B}
 = \int_B H(x,\6y) 
\end{equation} 
for all Borel sets $B\subset\partial\cD$. Note that $x\mapsto H(x,\6y)$ is 
$\cL$-harmonic, i.e., it satisfies $\cL H=0$ in $\cD$. The uniform ellipticity
assumption implies that for all $x\in\cD$, 
\begin{equation}
 H(x,\6y) = h(x,y)\6y
\end{equation} 
admits a density $h$ with respect to the arclength (one-dimensional surface
measure) $\6y$, which is smooth wherever the boundary is smooth. This has been
shown, e.g., in \cite{BenArous_Kusuoka_Stroock_1984} (under a weaker
hypoellipticity condition). 

We now derive some bounds on the magnitude of oscillations of $h$, based on
Harnack inequalities.

\begin{lemma}
\label{lem_Harnack0}
For any set $\cD_{0}$ such that its closure satisfies $\overline\cD_{0}\subset\cD$, there exists a constant $C$, independent of $\sigma$, such that 
\begin{equation}
 \label{Harnack0}
 \frac{\displaystyle\sup_{x\in\cD_{0}} h(x,y)}
 {\displaystyle\inf_{x\in\cD_{0}} h(x,y)}
 \leqs \e^{C/\sigma^2}
 \end{equation}
holds for all $y\in\partial\cD$. 
\end{lemma}
\begin{proof}
Let $\cB$ be a ball of radius $R=\sigma^2$ contained in $\cD_{0}$. 
By~\cite[Corollary 9.25]{Gilbarg_Trudinger}, we have for any $y\in\partial\cD$ 
\begin{equation}
\label{Harnack0:1} 
 \sup_{x\in\cB} h(x,y) \leqs C_0 \inf_{x\in\cB} h(x,y)\;,
\end{equation} 
where the constant $C_0\geqs1$ depends only on the ellipticity
constant of $g$ and on $\nu R^2$, where the parameter $\nu$ is an upper bound
on $(\norm{f}/\sigma^2)^2$. Since $R=\sigma^2$, $C_0$ does not depend
on~$\sigma$. Consider now two points $x_1,x_2\in\cD$. They can be joined by a
sequence of $N=\intpartplus{\norm{x_2-x_1}/\sigma^2}$ overlapping balls of
radius $\sigma^2$. Iterating the bound~\eqref{Harnack0:1}, we obtain 
\begin{equation}
 h(x_2,y) \leqs C_0^N h(x_1,y) 
 = \e^{(\log C_0) \intpartplus{\norm{x_2-x_1}/\sigma^2}} h(x_1,y)\;,
\end{equation} 
which implies the result. 
\end{proof}

\begin{lemma}
\label{lem_Harnack} 
Let $\cB_r(x)$ denote the ball of radius $r$ centred in $x$, and let $\cD_{0}$ be such that its closure satisfies $\overline\cD_{0}\subset\cD$. Then, 
for any $x_0\in\cD_{0}, y\in\partial\cD$ and $\eta>0$, one can find a constant
$r=r(y,\eta)$, independent of $\sigma$, such that 
\begin{equation}
 \label{Harnack}
 \sup_{x\in\cB_{r\sigma^2}(x_0)} h(x,y)
 \leqs (1+\eta) \inf_{x\in\cB_{r\sigma^2}(x_0)} h(x,y)\;.
 \end{equation} 
\end{lemma}
\begin{proof}
Let $r_0$ be such that $\cB_{r_0\sigma^2}(x_0)\subset\cD_{0}$, and write
$R_0=r_0\sigma^2$. Since $h$ is harmonic and positive, we can apply the
Harnack estimate~\cite[Corollary 9.24]{Gilbarg_Trudinger}, showing that for any
$R\leqs R_0$, 
\begin{equation}
 \osc_{\cB_R(x_0)} h 
 \defby \sup_{x\in\cB_R(x_0)} h(x,y) - \inf_{x\in\cB_R(x_0)} h(x,y)
 \leqs C_1 \biggpar{\frac{R}{R_0}}^\alpha \osc_{\cB_{R_0}(x_0)} h\;,
\end{equation} 
where, as in the previous proof, the constants $C_1\geqs1$ and $\alpha>0$ do not
depend on $\sigma$. By~\cite[Corollary 9.25]{Gilbarg_Trudinger}, we also have 
\begin{equation}
 \sup_{x\in\cB_{R_0}(x_0)} h(x,y) \leqs 
 C_2 \inf_{x\in\cB_{R_0}(x_0)} h(x,y)\;,
\end{equation} 
where again $C_2>1$ does not depend on $\sigma$. Combining the two
estimates, we obtain 
\begin{equation}
 \frac{\displaystyle\sup_{x\in\cB_R(x_0)}h(x,y)}
 {\displaystyle\inf_{x\in\cB_R(x_0)} h(x,y)} -1 
 \leqs \frac{\displaystyle \osc_{\cB_R(x_0)} h}{\displaystyle
\inf_{x\in\cB_{R_0}(x_0)} h(x,y)}
 \leqs C_1 \biggpar{\frac{R}{R_0}}^\alpha (C_2 - 1)\;.
\end{equation} 
The result thus follows by taking $r=R/\sigma^2$, where 
$R=R_0[\eta/(C_1(C_2-1))]^{1/\alpha}$. 
\end{proof}


\subsection{Random Poincar\'e maps}
\label{ssec_mpm}

Consider now an SDE of the form~\eqref{SDE_hm}, where $x=(\ph,r)$ and $f$ and
$g$ are periodic in $\ph$, with period $1$. Consider the domain 
\begin{equation}
\label{rpm01} 
 \cD = \bigsetsuch{(\ph,r)}{-M<\ph<1\;,\; a<r<b}\;,
\end{equation} 
where $a<b$, and $M$ is some (large) integer. We have in mind drift terms with
a positive $\ph$-component, so that sample paths are very unlikely to leave
$\cD$ through the segment $\ph=-M$. 

Given an initial condition $x_0=(0,r_0)\in\cD$, we can define 
\begin{equation}
\label{rpm02} 
 k(r_0,r_1) = h((0,r_0),(1,r_1))\;,
\end{equation} 
where $h(x,y)\6y$ is the harmonic measure. Then by periodicity of $f$ and $g$
and the strong Markov property, $k$ defines a Markov chain on $E=[a,b]$, keeping
track of the value $R_n$ of $r_t$ whenever $\ph_t$ first reaches $n\in\N$. This
Markov chain is substochastic because we only take into account paths reaching
$\ph=1$ before hitting any other part of the boundary of $\cD$. In other words,
the Markov chain describes the process killed upon $r_t$ reaching $a$ or $b$
(or $\ph_t$ reaching $-M$). 

We denote by $K_n(x,y)$ the $n$-step transition kernel, and by $k_n(x,y)$ its
density. Given an interval $A\subset E$, we write $K^{A}_n(x,y)=K_{n}(x,y)/K_{n}(x,A)$ for the $n$-step transition kernel for the Markov chain conditioned to stay in $A$, and $k^{A}_n(x,y)$ for the corresponding density.

\begin{prop}
\label{prop_positivity} 
Fix an interval $A\subset E$. 
For $x_1,x_2\in A$ define the integer stopping time 
\begin{equation}
 \label{rpm03} 
 N = N(x_1,x_2) =
\inf\bigsetsuch{n\geqs1}{\abs{X^{x_2}_n-X^{x_1}_n}<r_\eta\sigma^2}\;,
\end{equation} 
where $r_\eta=r(y,\eta)$ is the constant of Lemma~\ref{lem_Harnack} and $X^{x_{0}}_{n}$ denotes the Markov chain with transition kernel $K^{A}(x,y)$ and initial condition $x_{0}$.  The two Markov chains $X^{x_{1}}_{n}$ and $X^{x_{2}}_{n}$ are coupled in the sense that their dynamics is derived from the same realization of the Brownian motion, cf.~\eqref{SDE_hm}.

Let 
\begin{equation}
\rho_n = \sup_{x_1,x_2\in A}
\bigprob{N(x_1,x_2) > n }\;. 
\end{equation}
Then for any $n\geqs2$ and $\eta>0$, 
\begin{equation}
\label{rpm04} 
 \frac{\displaystyle\sup_{x\in A} k^{A}_n(x,y)}
 {\displaystyle\inf_{x\in A} k^{A}_n(x,y)}
 \leqs 1 + \eta + {\rho_{n-1}\e^{C/\sigma^2}}
\end{equation} 
holds for all $y\in A$, where $C$ does not depend on $\sigma$. 

\end{prop}
\begin{proof}
We decompose 
\begin{align}
\label{rpm05:1}
\nonumber 
 \bigprob{X^{x_1}_n\in\6y}
{} &{}=\sum_{k=1}^{n-1} \bigprob{X^{x_1}_n\in\6y\mid N=k}\bigprob{N=k} \\
&\phantom{{}={}}{} + \bigprob{X^{x_1}_n\in\6y\mid N>n-1} \bigprob{N>n-1 }\;.
\end{align} 
Let $k_{n}^{(2)}((x_{1},x_{2}),(z_{1},z_{2})\mid N=k)$ denote the conditional joint density for a transition for $(X^{x_{1}}_{l},X^{x_{2}}_{l})$ in $n$  steps from $(x_{1},x_{2})$ to $(z_{1},z_{2})$, given $N=k$. Note that this density is concentrated on the set $\set{\abs{z_{2}-z_{1}}< r_{\eta}\sigma^{2}}$.
For $k=1,\dots,n-1$ and any measurable $B\subset E$, we use Lemma~\ref{lem_Harnack} to estimate
\begin{align}
\nonumber
 &\bigprob{X^{x_1}_n\in B\mid N=k} \\
 \nonumber
 &\qquad{}= 
 \int_A\int_A \bigprob{X^{z_1}_{n-k}\in B} k_{n}^{(2)}((x_{1},x_{2}),(z_{1},z_{2})\mid N=k)\,\6z_{2}\6z_{1}
\\
\nonumber
 &\qquad{}\leqs 
(1+\eta)  \int_{A}\int_A \bigprob{X^{z_2}_{n-k}\in B} k_{n}^{(2)}((x_{1},x_{2}),(z_{1},z_{2})\mid N=k)\,\6z_{2}\6z_{1}
\\
&\qquad{}= (1+\eta)  \bigprob{X^{x_2}_n\in B\mid N=k} \;.
 \label{rpm05:2} 
\end{align} 
Writing $k_{n-1}(x_{1},z_{1}\mid N>n-1)$ for the conditional $(n-1)$-step
transition density of $X^{x_{1}}_{l}$,
the last term in~\eqref{rpm05:1} can be bounded by 
\begin{align}
 \label{rpm05:3}
 \nonumber
\bigprob{X^{x_{1}}_{n}\in B\mid N>n-1}
\nonumber
 {}&{} = 
 \int_A \bigprob{X^{z_{1}}_{1}\in B} k_{n-1}(x_{1},z_{1}\mid N>n-1) \,\6z_{1} \\
  \nonumber
 {}&{} \leqs \sup_{z_{1}\in A} \bigprob{X^{z_{1}}_{1}\in B} \bigprob{X^{x_{1}}_{n-1}\in E \mid N>n-1} \\
{}&{}\leqs \sup_{z_{1}\in A} \bigprob{X^{z_{1}}_{1}\in B}
\;.
\end{align} 
We thus have 
\begin{equation}
\label{rpm05:4}
\bigprob{X^{x_1}_n\in\6y} \leqs (1+\eta) \bigprob{X^{x_2}_n\in\6y}
+ \rho_{n-1} \sup_{z_1\in A}\bigprob{X^{z_1}_1\in\6y}\;.
\end{equation}
On the other hand, we have 
\begin{equation}
 \bigprob{X^{x_1}_n\in\6y} \geqs \bigprob{X^{x_1}_{n-1}\in A}
 \inf_{z_1\in A}\bigprob{X^{z_1}_1\in\6y}\;.
\end{equation} 
Combining the upper and lower bound, we get 
\begin{equation}
\frac{\displaystyle\sup_{x\in A} k^{A}_n(x,y)}
 {\displaystyle\inf_{x\in A} k^{A}_n(x,y)}
 \leqs 1 + \eta + {\rho_{n-1}}\,
 \frac{\displaystyle\sup_{z\in A} k^{A}(z,y)}
 {\displaystyle\inf_{z\in A} k^{A}(z,y)}
\;.
\end{equation} 
Hence the result follows from Lemma~\ref{lem_Harnack0}. 
\end{proof}


\section{Sample-path estimates}
\label{sec_sp}


\subsection{The principal eigenvalue $\lOu$}
\label{ssec_lu}

We consider in this section the system 
\begin{align}
\nonumber
\6r_t &= \bigbrak{\lambda_{+} r_t + b_r(r_t,\ph_t)} \6t 
+ \sigma g_r(r_t,\ph_t) \6W_t\;, \\
\6\ph_t &= \Bigbrak{\frac{1}{T_+} + b_\ph(r_t,\ph_t)} \6t 
+ \sigma g_\ph(r_t,\ph_t) \6W_t\;,
\label{lu01} 
\end{align}
describing the dynamics near the unstable orbit. We have redefined $r$ in
such a way that the unstable orbit is located in $r=0$, and that the stable
orbit lies in the region $\set{r>0}$. Here $\set{W_t}_t$ is a
$k$-dimensional standard Brownian motion, $k\geqs 2$, and $g=(\transpose{g_r},
\transpose{g_\ph})$
satisfies a
uniform ellipticity condition. The functions $b_r$, $b_\ph$, $g_r$ and $g_\ph$
are periodic in $\ph$ with period~$1$ and the nonlinear drift terms
satisfy $\abs{b_r(r,\ph)}, \abs{b_\ph(r,\ph)} \leqs M r^2$. 

Note that in first approximation, $\ph_t$ is close to $t/T_+$. 
Therefore we start by considering the linear process $r^0_t$ defined by 
\begin{equation}
 \label{lu02}
 \6r^0_t = \lambda r^0_t \6t + \sigma g_0(t) \6W_t\;,
\end{equation} 
where $g_0(t) = g_r(0, t/T_+)$, and $\lambda$ will be chosen close to
$\lambda_+$. 

\begin{prop}[Linear system]
\label{prop_lu_linear}
Choose a $T>0$ and fix a small constant $\delta>0$. Given $r_0\in(0,\delta)$ and
an interval $A\subset(0,\delta)$, define 
\begin{equation}
\label{lu03} 
 P(r_0,A) = \bigprobin{r_0}{0 < r^0_t < \delta \; \forall
t\in[0,T], r_T\in A}\;,
\end{equation} 
and let 
\begin{equation}
 \label{lu04}
 v_t = \int_0^t \e^{-2\lambda s} g_0(s)\transpose{g_0(s)} \6s
 \qquad \text{for } t\in[0,T]
 \;.
\end{equation} 
\begin{enum}
 \item 	\emph{Upper bound:} For any $T>0$, 
 \begin{equation}
  \label{lu05} 
  P(r_0,(0,\delta)) \leqs \frac{1}{\sqrt{2\pi}}
\frac{\delta^2r_0}{\sigma^3v_T^{3/2}}
  \e^{-r_0^2/2\sigma^2v_T} \e^{-2\lambda T}
  \biggbrak{1 + \biggOrder{\frac{r_0^2\e^{-2\lambda T}}{\sigma^4 v_T^2}}}\;.
 \end{equation} 
 \item 	\emph{Lower bound:} Assume $A=[\sigma a, \sigma b]$ for two constants
$0<a<b$. Then there exist constants $C_0, C_1, c>0$, depending only on $a, b,
\lambda$, such that for any $r_0\in A$ and $T\geqs 1$, 
\begin{equation}
 \label{lu06}
  P(r_0,A) \geqs \biggpar{C_0 -C_1 T \frac{\e^{-c\delta^2/\sigma^2}}{\delta^2}}
\e^{-2\lambda
T}\;.
\end{equation} 
\end{enum}
\end{prop}
\begin{proof}
We shall work with the rescaled process $z_t=\e^{-\lambda t}r^0_t$, which
satisfies 
\begin{equation}
 \label{lu07:01}
 \6z_t = \e^{-\lambda t}\sigma g_0(t) \6W_t\;.
\end{equation} 
Note that $z_t$ is Gaussian with variance $\sigma^2 v_t$. Using Andr\'e's
reflection principle, we get 
\begin{align}
\nonumber
P(r_0,(0,\delta)) 
&\leqs \bigprobin{r_0}{z_t > 0 \; \forall t\in[0,T], 
0 < z_T < \delta \e^{-\lambda T}} \\
\nonumber
&= \bigprobin{r_0}{0 < z_T < \delta \e^{-\lambda T}}
- \bigprobin{-r_0}{0 < z_T < \delta \e^{-\lambda T}} \\
\nonumber
&= \frac{1}{\sqrt{2\pi\sigma^2v_T}}
\int_0^{\delta \e^{-\lambda T}} 
\Bigbrak{\e^{-(r_0-z)^2/2\sigma^2v_T} 
- \e^{-(r_0+z)^2/2\sigma^2v_T}} \6z \\
&\leqs \frac{2}{\sqrt{2\pi\sigma^2v_T}}
\e^{-r_0^2/2\sigma^2v_T} \int_0^{\delta \e^{-\lambda T}}  
\sinh\biggpar{\frac{r_0z}{\sigma^2v_T}} \6z\;, 
\label{lu07:02} 
\end{align}
and the upper bound~\eqref{lu05} follows by using
$\cosh(u)-1=\frac12u^2+\Order{u^4}$. 

To prove the lower bound, we introduce the notations $\tau_0$ and $\tau_\delta$
for the first-hitting times of $r_t$ of $0$ and $\delta$. Then
we can write 
\begin{equation}
 \label{lu07:03}
 P(r_0,A) = 
 \bigprobin{r_0}{\tau_0>T, r_T\in A} 
 - \bigprobin{r_0}{\tau_\delta < T < \tau_0\;,\;r_T\in A}\;.
\end{equation} 
The first term on the right-hand side can be bounded below by a similar
computation 
as for the upper bound. Using that $r_0$ is of order $\sigma$, that $v_T$ has
order $1$ for $T\geqs1$, and taking into account the different domain of
integration, one obtains a lower bound $C_0\e^{-2\lambda T}$. As for the second
term on the right-hand side, it can be rewritten as 
\begin{equation}
 \label{lu07:04}
 \Bigexpecin{r_0}{\indexfct{\tau_\delta < T \wedge \tau_0}
 \bigprobin{\delta}{\tau_0 > T-\tau_\delta\;,\; 
 r_{T-\tau_\delta}\in A}}\;.
\end{equation}
By the upper bound~\eqref{lu05}, the probability inside the expectation is
bounded by a constant times $\e^{-2\lambda(T-\tau_\delta)} \e^{-c
\delta^2/\sigma^2}/\delta^2$. 
It remains to estimate $\expecin{r_0}{\indexfct{\tau_\delta < T \wedge \tau_0}
\e^{2\lambda\tau_\delta}}$. Integration by parts and another
application of~\eqref{lu05} show that this term is
bounded by a constant times $T$, and the lower bound is proved. 
\end{proof}

\goodbreak 

\begin{remark}
\label{rem_lu_linear}
\hfill
\begin{enum}
\item 	The upper bound~\eqref{lu05} is maximal for $r_0=\sigma\sqrt{v_T}$,
with a value of order $(\delta^2/\sigma^2v_T)\e^{-2\lambda T}$. 
\item 	Applying the reflection principle at a level $-a$ instead of $0$, one
obtains 
\begin{equation}
 \label{lu08}
 \bigprobin{r_0}{-a\e^{\lambda t} < r^0_t < \delta \; \forall t\in[0,T]}
\leqs C_0 \frac{(\delta\e^{-\lambda T}+a)^2}{\sigma^2 v_T}
\end{equation} 
for some constant $C_0$ (provided the higher-order error terms are small). 
\end{enum}
\end{remark}

We will now extend these estimates to the general nonlinear
system~\eqref{lu01}. We first show that $\ph_t$ does not differ much from
$t/T_+$ on rather long timescales. To ease the notation, given $h,h_1>0$ ,we
introduce two stopping times
\begin{align}
\nonumber
\tau_h &= \inf\setsuch{t>0}{r_t\geqs h}\;, \\
\tau_\ph &= \inf\biggsetsuch{t>0}{\Bigabs{\ph_t - \frac{t}{T_+}} \geqs M
\bigpar{h^2 t + h_1\,}}\;.
\label{lu08B}
\end{align}

\begin{prop}[Control of the diffusion along $\ph$]
\label{prop_lu_phi}
There is a constant $C_1$, depending only on the ellipticity constants
of the diffusion terms, such that 
\begin{equation}
 \label{lu09}
 \bigprobin{(r_0,0)}{\tau_\ph < \tau_h\wedge T}
 \leqs \e^{-h_1^2/(C_1 h^2\sigma^2T)}
\end{equation} 
holds for all $T,\sigma>0$ and all $h, h_1>0$. 
\end{prop}
\begin{proof}
Just note that $\eta_t = \ph_t - t/T_+$ is given by 
\begin{equation}
\label{lu10} 
 \eta_t = \int_0^t b_\ph(r_s,\ph_s)\6s + \sigma \int_0^t g_\ph(r_s,\ph_s)\6Ws\;.
\end{equation} 
For $t<\tau_h$, the first term is bounded by $Mh^2t$, while the probability that
the second one becomes large can be bounded by the Bernstein-type estimate
Lemma~\ref{lem_Bernstein}. 
\end{proof}

In the following, we will set $h_1=\sqrt{h^3T}$. In that case, 
$h^2 t + h_1 \leqs h(1+2hT)$, and the right-hand side of~\eqref{lu09}
is bounded by $\e^{-h/(C_1\sigma^2)}$. All results below hold for all $\sigma$ sufficiently small, as indicated by the $\sigma$-dependent error terms. 

\begin{prop}[Upper bound on the probability to stay near the unstable orbit]
\label{prop_lu_upperbound} 
Let $h=\sigma^\gamma$ for some $\gamma\in(1/2,1)$, and let $\mu>0$ satisfy
$(1+2\mu)/(2+2\mu) > \gamma$. 
Then for any $0<r_0<h$ and all $0<T\leqs 1/h$,
\begin{equation}
 \label{lu11} 
 \bigprobin{r_0}{0 < r_t < h \; \forall t\in[0,T\wedge\tau_\ph]} 
 \leqs \frac{1}{\sigma^{2\mu(1-\gamma)}}
 \exp\biggset{-\lambda_{+} T \biggbrak{\frac{2\mu}{1+\mu} -
\biggOrder{\frac{1}{\logs}}}}\;. 
\end{equation} 
\end{prop}
\begin{proof}
The proof is very close in spirit to the proof of~\cite[Theorem~3.2.2]{BGbook}, 
so that we will only give the main ideas. The principal difference is that we
are interested in the exit from an asymmetric interval $(0,h)$, which yields an
exponent close to $2\lambda_+$ instead of $\lambda_+$ for a symmetric interval
$(-h,h)$. To ease the notation, we will write $\lambda$ instead of $\lambda_{+}$ throughout the proof.

We introduce a partition of $[0,T]$ into intervals of equal length
$\Delta/\lambda$, for a $\Delta$ to be chosen below. Then the Markov
property implies that the probability~\eqref{lu11} is bounded by 
\begin{equation}
 \label{lu12:1}
 q(\Delta)^{-1} \exp\biggset{-\lambda T \frac{\log(q(\Delta)^{-1})}{\Delta}}\;,
\end{equation} 
where $q(\Delta)$ is an upper bound on the probability to leave $(0,h)$ on a
time interval of length $\Delta/\lambda$. We thus want to show that 
$\log(q(\Delta)^{-1})/\Delta$ is close to $2$ for a suitable choice of
$\Delta$. 

We write the equation for $r_t$ in the form 
\begin{equation}
 \label{lu12:2}
 \6r_t = \bigbrak{\lambda r_t + b(r_t,\ph_t)}\6t + \sigma g_0(t)\6W_t +
 \sigma g_1(r_t,\ph_t,t)\6W_t\;.
\end{equation} 
Note that for $\abs{r_t}<h$ and $t<\tau_\ph\wedge T$, we may assume that 
$g_1(r_t,\ph_t,t)$ has order $h+h^2T$, which has in fact order $h$ since we
assume $T\leqs 1/h$. 
Introduce the Gaussian processes 
\begin{equation}
 \label{lu12:3}
 r^\pm_t = r_0 \e^{\lambda^\pm t} + \sigma \e^{\lambda^\pm t}
 \int_0^t \e^{-\lambda^\pm s}g_0(s)\6W_s\;,
\end{equation} 
where $\lambda^\pm = \lambda\pm Mh$. 
Applying the comparison principle to $r_t-r^+_t$, we have 
\begin{equation}
 \label{lu12:4}
 r^-_t + \sigma \e^{\lambda^- t}\cM^-_t \leqs r_t \leqs r^+_t + \sigma
\e^{\lambda^+ t}\cM^+_t
\end{equation} 
as long as $0<r_t\leqs h$, 
where $\cM^\pm_t$ are the martingales 
\begin{equation}
 \label{lu12:5}
 \cM^\pm_t = \int_0^t \e^{-\lambda^\pm s}
g_1(r_s,\ph_s,s)\6W_s\;.
\end{equation} 
We also have the relation 
\begin{equation}
 \label{lu12:6}
 r^+_t = \e^{2Mht}r^-_t + \sigma \e^{\lambda^+ t} \cM^0_t
 \qquad
 \text{where}
 \quad
 \cM^0_t = \int_0^t \bigbrak{\e^{\lambda^+ s} - \e^{\lambda^- s}}g_0(s)\6W_s\;.
\end{equation} 
Using It\^o's isometry, one obtains that $\cM^0_t$ has a variance of order
$h^2$. 
This, as well as Lemma~\ref{lem_Bernstein} in the case of $\cM^\pm_t$, shows
that 
\begin{equation}
 \label{lu12:8}
 \Bigprob{\sup_{0<s<t}\abs{\sigma\e^{\lambda^+s}\cM^{0,\pm}_s} > H} 
 \leqs 
 \exp\biggset{-\frac{H^2}{2C_1h^2\sigma^2\e^{2\lambda^+t}}}
\end{equation} 
for some constant $C_1$. 
Combining~\eqref{lu12:4} and~\eqref{lu12:6}, we obtain that $0<r_t<h$ implies 
\begin{equation}
 \label{lu12:7}
-\sigma  \e^{\lambda^+ t}\cM^+_t < r^+_t < \e^{2Mht}
\bigbrak{h+\sigma\e^{\lambda^-
t}\cM^-_t} - \sigma \e^{\lambda^+ t} \cM^0_t\;.
\end{equation} 
The probability we are looking for is thus bounded by 
\begin{equation}
 \label{lu12:9} 
 q(\Delta) = \Bigprob{-H < r^+_t <\e^{2Mht}\bigbrak{h+ H} + H \;\forall
t\in[0,\Delta/\lambda]} + 3
\exp\biggset{-\frac{H^2}{2C_1h^2\sigma^2\e^{2\lambda^+\Delta/\lambda}}}\;.
\end{equation} 
The first term on the right-hand side can be bounded using~\eqref{lu08} with
$a=H$, yielding 
\begin{equation}
 \label{lu12:10}
 q(\Delta) \leqs \frac{C_0}{\sigma^2}
 \biggbrak{(\e^{2Mh\Delta/\lambda}\bigbrak{h+ H} + H)\e^{-\Delta}+H}^2
 + 3 \exp\biggset{-\frac{H^2}{2C_1h^2\sigma^2\e^{2\lambda^+\Delta/\lambda}}}\;.
\end{equation} 
We now make the choices 
\begin{equation}
 \label{lu12:11}
 H = \e^{-\Delta}h
 \qquad
 \text{and}
 \qquad
 \Delta = \frac{1+\mu}{2} \log \biggpar{1+\mu+\frac{h^2}{\sigma^2}}\;.
\end{equation} 
Substituting in~\eqref{lu12:10} and carrying out computations similar to those
in~\cite[Theorem~3.2.2]{BGbook} yields $\log(q(\Delta)^{-1})/\Delta \geqs
2\mu/(1+\mu) - \Order{1/\logs}$, and hence the result. 
\end{proof}

The estimate~\eqref{lu11} can be extended to the exit from a neighbourhood of
order $1$ of the unstable orbit, using exactly the same method as
in~\cite[Section~D]{BGK12}: 

\begin{prop}
\label{prop_lu_upperbound2} 
Fix a small constant $\delta>0$. Then for any 
$\kappa < 2$, there exist constants $\sigma_0,\alpha,C>0$ and $0<\nu<2$ such that  
\begin{equation}
 \label{lu14} 
 \bigprobin{r_0}{0 < r_t < \delta \; \forall t\in[0,T]} 
 \leqs \frac{C}{\sigma^{\alpha}} \e^{-\kappa\lambda_{+} T}
\end{equation} 
holds for all $r_0\in(0,\delta)$, all $\sigma<\sigma_0$ and all
$T\leqs\sigma^{-\nu}$. 
\end{prop}
\begin{proof}
The proof follows along the lines of~\cite[Sections~D.2 and D.3]{BGK12}. The
idea is to show that once sample paths have reached the level $h=\sigma^\gamma$,
they are likely to reach level $\delta$ after a relatively short time, without
returning below the level $h/2$. To control the effect of paths which switch
once or several times between the levels $h$ and $h/2$ before leaving
$(0,\delta)$, one uses Laplace transforms. 

Let $\tau_1$ denote the first-exit time of $r_t$ from $(0,h)$, where we set
$\tau_1=T$ if $r_t$ remains in $(0,h)$ up to time $T$. Combining
Proposition~\ref{prop_lu_phi} with $h_1=\sqrt{h^3t}$ and
Proposition~\ref{prop_lu_upperbound}, we obtain 
\begin{equation}
 \label{lu14:1}
 \bigprobin{r_0}{\tau_1 > t}
 \leqs \frac{1}{\sigma^{2\mu(1-\gamma)}} \e^{-\kappa_1\lambda_{+} t}
 + \e^{-1/(C_1\sigma^{2-\gamma})}
 \quad\forall t\in[0,T]
 \;,
\end{equation} 
where $\kappa_1=2\mu/(1+\mu) - \Order{1/\logs}$. The first term 
dominates the second one as long as $\nu < 2-\gamma$. Thus the Laplace
transform $\expec{\e^{u(\tau_1\wedge T)}}$ exists for all
$u<1/(\kappa_1\lambda_{+})$. 

Let $\tau_2$ denote the first-exit time of $r_t$ from $(h/2,\delta)$. As
in~\cite[Proposition~D.4]{BGK12}, using the
fact that the drift term is bounded below by a constant times $r$,
that $\set{\tau_2>t}\subset\set{r_t<\delta}$, 
an endpoint estimate and the Markov property to restart the process at
times which are multiples of $\logs$, we obtain 
\begin{equation}
 \label{lu14:2}
 \bigprobin{h}{\tau_2 > t} \leqs \exp\biggset{-C_2
\frac{t}{\sigma^{2(1-\gamma)}\logs}}
 \quad\forall t\in[0,T]
 \;
\end{equation} 
for some constant $C_2$. Therefore the Laplace transform
$\expec{\e^{u(\tau_2\wedge T)}}$ 
exists for all $u$ of order $1/(\sigma^{2(1-\gamma)}\logs)$. In
addition, one can show that the probability that sample paths starting at level
$h$ reach $h/2$ before $\delta$ satisfies 
\begin{equation}
 \label{lu14:3}
 \bigprobin{h}{\tau_{h/2} < \tau_\delta} 
 \leqs 2 \exp \biggset{-C\frac{h^2}{\sigma^2}}\;,
\end{equation} 
which is exponentially small in $1/\sigma^{2(1-\gamma)}$. 

We can now use~\cite[Lemma~D.5]{BGK12} to estimate the Laplace transform of
$\tau=\tau_0\wedge\tau_\delta\wedge T$, and thus the decay of
$\prob{\tau>t}$ via the Markov
inequality. Given $\kappa=2-\epsilon$, we first choose $\mu$ and $\sigma_0$ such
that $\kappa_1\leqs2-\epsilon/2$. This allows to estimate
$\expec{\e^{u\tau}}$ for $u=\kappa_1-\epsilon/2$ to get the desired
decay, and determines $\alpha$. The choice of $\mu$ also determines $\gamma$ and
thus $\nu$. 
\end{proof}

\begin{prop}[Lower bound on the probability to stay near the unstable orbit]
\label{prop_lu_lowerbound} 
Let $h=\sigma^\gamma$ for some $\gamma\in(1/2,1)$, and let $A=[\sigma a,\sigma
b]$ for constants $0<a<b$. Then there exists a constant $C$ such that 
\begin{equation}
 \label{lu15}
 \bigprobin{r_0}{0 < r_t < h \; \forall t\in[0,T]\;,\; r_T\in A}
 \geqs C 
 \exp\biggset{-2\lambda_{+} T \biggbrak{1 +
\biggOrder{\frac{1}{\logs}}}}
\end{equation} 
holds for all $r_0\in A$ and all $T\leqs 1/h$. 
\end{prop}
\begin{proof}
Consider again a partition of $[0,T]$ into intervals of length
$\Delta/\lambda_{+}$, and let $q(\Delta)$ be a lower bound on 
\begin{equation}
 \label{lu15:1}
 \bigprobin{r_0}{0 < r_t < h \; \forall t\in[0,\Delta/\lambda_{+}]\;,\; 
 r_{\Delta/\lambda_{+}}\in A}
\end{equation} 
valid for all $r_0\in A$. By comparing, as in the proof of
Proposition~\ref{prop_lu_upperbound}, $r_t$ with solutions of linear equations,
and using the lower bound of Proposition~\ref{prop_lu_linear}, we obtain 
\begin{equation}
 \label{lu15:2}
 q(\Delta) = C_1\e^{-2\Delta} - C_2 \e^{-c/(\sigma^2\e^{4\Delta})}
\end{equation} 
for constants $C_1,C_2,c>0$, 
where the second term bounds the probability that the martingales
$\sigma\cM^{0,\pm}_t$ exceed $H=\e^{-\Delta}h$ times an exponentially decreasing curve. 
By the Markov property, we can bound
the probability we are interested in below by the expression~\eqref{lu12:1}. The
result follows by choosing $\Delta=c_0\logs$ for a constant $c_0$. 
\end{proof} 

We can now use the last two bounds to estimate the principal eigenvalue of the
Markov chain on $E=[0,2\delta]$ with kernel $\Ku$, describing the process killed
upon hitting either the unstable orbit at $r=0$ or level $r=2\delta$. 

\begin{theorem}[Bounds on the principal eigenvalue $\lOu$]
\label{thm_lOu} 
For any sufficiently small $\delta>0$, there exist constants $\sigma_0, c>0$ such that 
\begin{equation}
 \label{lu16}
 (1-c\delta^{2})\e^{-2\lambda_+T_+}
 \leqs \lOu \leqs
 (1+c\delta^{2})\e^{-2\lambda_+T_+}
\end{equation} 
holds for all $\sigma < \sigma_0$. 
\end{theorem}
\begin{proof}
We will apply Proposition~\ref{prop_estimate_l0}. In order to do so, we pick $n\in\N$  
such that 
\begin{equation}
 \label{lu16:1}
 T = \frac{nT_+}{1+M\delta^2T_+}
\end{equation} 
satisfies Proposition~\ref{prop_lu_upperbound2} and is of order $\sigma^{-\nu}$ with $\nu<2$. Proposition~\ref{prop_lu_phi} shows that
with probability larger than $1-\e^{-\delta/(C_1\sigma^{2-\gamma})}$, 
\begin{equation}
 \label{lu16:2}
 \ph_t \leqs \frac{t}{T_+} + M\delta^2T
 \qquad
 \text{for all } t\leqs T\wedge \tau_{h}
\end{equation} 
for $h=\sigma^{\gamma}$  as before, with $\gamma >\nu$. In particular, we have $\ph_T\leqs n$. Together with
Propositions~\ref{prop_lu_upperbound} and~\ref{prop_lu_upperbound2} applied for $\kappa=2-\delta^2$, this
shows that for any $r_0\in E$
\begin{align}
\label{lu16:3} 
\nonumber
 \Ku_{n}(r_0,E) 
\leqs  {}& \frac{C}{\sigma^\alpha}
 \exp\biggset{-(2-\delta^2)\lambda_+ \frac{nT_+}{1+M\delta^2T_+}} 
  + \e^{-1/(C_1\sigma^{2-\gamma})} \\
&{} +\frac1{\sigma^{2\mu(1-\gamma)}} \exp\biggset{-\lambda_{+} T \biggbrak{\frac{2\mu}{1+\mu} -
\biggOrder{\frac{1}{\logs}}}}\;.
\end{align} 
Using $\log(a+b+c)\leqs \log 3+ \max\set{\log a, \log b,\log c}$ and the fact that $\nu<2$, we obtain 
\begin{align}
 \label{lu16:4}
 \nonumber%
&\frac1n\log \Ku_{n}(r_0,E) \\
&\qquad{} \leqs 
 \max\biggset{
 -(2-\delta^2) \frac{\lambda_+T_+}{1+M\delta^2T_+},
 - \frac{2\mu}{1+\mu}\frac{\lambda_{+} T_{+}}{1+M\delta^{2}T_{+}},
  \frac 1n \bigOrder{\logs + \sigma^{-(2-\gamma)}}}
\end{align} 
Since $n$ has order $\sigma^{-\nu}$, we can make
$\sigma$ small enough for all error terms to be of order $\delta^2$. Choosing first $\mu$, then the other parameters, 
proves the upper bound. The proof of the lower bound is similar. It is based on
Proposition~\ref{prop_lu_lowerbound}, a basic comparison between $r_{T}$ and the value of $r_{t}$ at the time $t$ when $\varphi_{t}$ reaches $n$, and the lower bound in
Proposition~\ref{prop_estimate_l0}.
\end{proof}


\subsection{The first-hitting distribution when starting in the QSD $\piOu$}
\label{ssec_sq}

In this section, we consider again the system~\eqref{lu01} describing the
dynamics near the unstable orbit. Our aim is now to estimate the
distribution of first-hitting locations of the unstable orbit when starting in
the quasistationary distribution $\piOu$. 

Consider first the linear process $r^0_t$ introduced in~\eqref{lu02}. By the
reflection principle, the distribution function of $\tau^0$, the first-hitting
time of $0$, is given by 
\begin{equation}
 \label{sq01}
 \bigprobin{r_0}{\tau^0 \leqs t} = 
 2 \Phi \biggpar{-\frac{r_0}{\sigma \sqrt{v_t}}}\;,
\end{equation} 
where $v_t$ is defined in~\eqref{lu04}, and
$\Phi(x)=(2\pi)^{-1/2}\int_{-\infty}^x \e^{-y^2/2}\6y$ is the distribution
function of the standard normal law.
The density of $\tau^0$ can thus be written as  
\begin{equation}
 \label{sq02}
 f_0(t) = 
 \frac{g_0(t)\transpose{g_0(t)}\e^{-2\lambda t}}{\sqrt{2\pi}v_t^{3/2}}
 \frac{r_0}{\sigma} \e^{-r_0^2/(2\sigma^2v_t)}
 =
 \frac{D_{rr}(1,t/T_+)\e^{-2\lambda t}}{\sqrt{2\pi}v_t}
 F\biggpar{\frac{r_0}{\sigma \sqrt{v_t}}}\;,
\end{equation} 
where 
\begin{equation}
 \label{sq03}
 F(u) = u \e^{-u^2/2}\;.
\end{equation} 
Observe in particular that 
$v_t$ converges as $t\to\infty$ to a constant $v_\infty >0$, and that 
\begin{equation}
 \label{sq04}
 v_t = v_\infty - \Order{\e^{-2\lambda t}}\;.
\end{equation} 
The density $f_0(t)$ thus asymptotically behaves like a periodically modulated
exponential. 

The following result establishes a similar estimate for a coarse-grained version
of the first-hitting density of the nonlinear process. We set 
\begin{equation}
 \label{sq05}
 \tau=\inf\setsuch{t>0}{r_t=0}
\end{equation} 
and write $V(\ph) = v_{T_+\ph}$. 

\begin{prop}[Bounds on the first-hitting distribution starting from a point]
\label{prop_sq1}  
Fix constants $\Delta, T > 0$ and $0<\eps<1/3$. Then there
exist $\sigma_0, \gamma, \kappa>0$, depending on $\Delta, \delta, \eps$ and
$T$, such that for all $\sigma<\sigma_0$ and $\ph_0\in[1,T/T_+]$,
\begin{align}
\nonumber
\bigprobin{r_0}{\ph_\tau\in[\ph_0,\ph_0+\Delta]} 
 ={}& \frac{T_+}{\sqrt{2\pi}}
 \int_{\ph_0}^{\ph_0+\Delta}
 \frac{D_{rr}(1,\ph)\e^{-2\lambda_+ T_+\ph}}{V(\ph)}
 F \biggpar{\frac{r_0}{\sigma \sqrt{V(\ph)}}} \6\ph
 \,\bigbrak{1+\Order{\sigma^\gamma}} \\
 &{}+ \Order{\e^{-\kappa/\sigma^{2\eps}}}
 \label{sq06}
\end{align} 
holds for all $\sigma^{2-3\eps} < r_0 < \delta$.
Furthermore, 
\begin{equation}
 \label{sq07}
 \bigprobin{r_0}{\ph_\tau\in[\ph_0,\ph_0+\Delta]} =
\Order{\sigma^{1-3\eps}}
\end{equation} 
for $0\leqs r_0 \leqs \sigma^{2-3\eps}$. 
\end{prop}
\begin{proof}
We set $\ph_1=\ph_0+\Delta$, 
$h=\sigma^{1-\eps}$, $h_1=H=\sigma^{2-2\eps}$, and 
$h_2=M(h^2(T+1)+h_1)$. 
Let $r^\pm_t$ be the linear processes introduced in~\eqref{lu12:3}
and consider the events 
\begin{align}
 \label{sq08:1}
 \Omega_1 &= 
 \bigset{\ph_\tau \leqs \ph_0+\Delta}\;,  \\
 \Omega_2 &= \biggset{r_t\leqs h, \;
 \biggabs{\ph_t - \frac{t}{T_+}} \leqs h_2, \;
 r^-_t - H\e^{\lambda^-(t-T)} \leqs r_t \leqs r^+_t + H\e^{\lambda^+(t-T)}
\:\forall t\leqs \tau
 }\;.
 \nonumber
\end{align} 
Proposition~\ref{prop_lu_phi}, \eqref{lu12:4} and the estimates~\eqref{lu12:8}
and~\eqref{lu14:3} imply that there exists $\kappa>0$ such that 
\begin{equation}
 \label{sq08:2}
 \fP(\Omega_1\cap\Omega_2^c) \leqs 3\e^{-\kappa/\sigma^{2\eps}}\;.
\end{equation} 
Define the stopping times 
\begin{equation}
 \label{sq08:3}
 \tau^0_\pm 
 = \inf\bigsetsuch{t>0}{r^\pm_t = \mp H\e^{\lambda^\pm
(t-T)}}\;.
\end{equation} 
Since the processes $r^\pm_t - r_0\e^{\lambda^\pm t}$ satisfy linear equations
similar to~\eqref{lu02}, we can compute the densities of $\tau^0_\pm$, in
perfect analogy with~\eqref{sq02}. Scaling by $T_+$ for later convenience, we
obtain that the densities of $\tau^0_\pm/T_+$ are given by 
\begin{equation}
 \label{sq08:4} 
 f_\pm(\ph) = \frac{T_+}{\sqrt{2\pi}} 
 \frac{D_{rr}(1,\ph)\e^{-2\lambda^\pm T_+\ph}}{V(\ph)} 
 F \biggpar{\frac{r_0\pm H\e^{-\lambda^\pm T}}{\sigma \sqrt{V(\ph)}}}\;.
\end{equation} 
By definition of $\Omega_1$, $\ph_\tau \leqs \ph_0$
implies $\tau^0_- \leqs T_+(\ph_0+h_2)$ and $\tau^0_+ \leqs T_+(\ph_0-h_2)$
implies $\ph_\tau \leqs \ph_0$ on $\Omega_1$.
Therefore, we have 
\begin{align}
\nonumber
 \bigprobin{r_0}{\ph_\tau\in[\ph_0,\ph_1]}
 \leqs{}& \bigprobin{r_0}{\tau^0_-/T_+ \leqs \ph_1 + h_2}
 - \bigprobin{r_0}{\tau^0_+/T_+ \leqs \ph_0 - h_2} + \fP(\Omega_1\cap\Omega_2^c)
\\
 \label{sq08:5}
 ={}& \int_{\ph_0 - h_2}^{\ph_1 + h_2} f_-(\ph)\6\ph \\
 &{}+ 2\Phi\biggpar{-\frac{r_0-H\e^{-\lambda^-T}}{\sigma \sqrt{V(\ph_0-h_2)}}} 
 - 2\Phi\biggpar{-\frac{r_0+H\e^{-\lambda^+T}}{\sigma \sqrt{V(\ph_0-h_2)}}} 
 + \fP(\Omega_1\cap\Omega_2^c)\;, 
\nonumber
\end{align} 
and, similarly, 
\begin{align}
 \label{sq08:5b}
 \bigprobin{r_0}{\ph_\tau\in[\ph_0,\ph_1]}
 \geqs{}& \int_{\ph_0 + h_2}^{\ph_1 - h_2} f_+(\ph)\6\ph \\
 &{}- 2\Phi\biggpar{-\frac{r_0-H\e^{-\lambda^-T}}{\sigma \sqrt{V(\ph_0+h_2)}}} 
 + 2\Phi\biggpar{-\frac{r_0+H\e^{-\lambda^+T}}{\sigma \sqrt{V(\ph_0+h_2)}}} 
 - \fP(\Omega_1\cap\Omega_2^c)\;.
\nonumber
\end{align} 
We now distinguish three cases, depending on the value of $r_0$. 

\begin{enum}
\item 	{\it Case $r_0 > \sigma^{1-\eps}$.} 
All terms on the right-hand side of~\eqref{sq08:5} and~\eqref{sq06} are of order
$\e^{-\kappa/\sigma^{2\eps}}$ for some $\kappa>0$, so that the result follows
immediately. 

\item 	{\it Case $\sigma^{2-3\eps} \leqs r_0 \leqs \sigma^{1-\eps}$.} Here it
is useful to notice that for any $\mu>0$ and all $u$, 
\begin{equation}
 \label{sq08:6} 
 \frac{F(\mu u)}{F(u)} = \mu \e^{-(\mu^2-1)u^2/2}
 = 1 + \bigOrder{(\mu-1)(1+u^2)}\;.
\end{equation} 
Applying this with $\mu = (r_0-H\e^{-\lambda^- T})/r_0 = 1 + \Order{H/r_0}$
shows that 
\begin{equation}
 \label{sq08:7}
 F\biggpar{\frac{r_0-H\e^{-\lambda^- T}}{\sigma\sqrt{V(\ph)}}}
 = F\biggpar{\frac{r_0}{\sigma\sqrt{V(\ph)}}}
 \bigbrak{1 + \Order{\sigma^\eps} + \Order{\sigma^{1-3\eps}}}\;,
\end{equation}
where the two error terms bound $H/r_0$ and
$(H/r_0)(r_0^2/\sigma^2)$, respectively. 
This shows that 
\begin{align}
\nonumber
 \int_{\ph_0 - h_2}^{\ph_1 + h_2} f_-(\ph)\6\ph 
 ={}& \frac{T_+}{\sqrt{2\pi}} \int_{\ph_0 - h_2}^{\ph_1 + h_2} 
 \frac{D_{rr}(1,\ph)\e^{-2\lambda_+ T_+\ph}}{V(\ph)} 
 F \biggpar{\frac{r_0}{\sigma \sqrt{V(\ph)}}} \6\ph
 \\
 &{}\times  
 \bigbrak{1 + \Order{\sigma^\eps} + \Order{\sigma^{1-3\eps}}}
 \label{sq08:8}
\end{align}
(note that replacing $\lambda^-$ by $\lambda_+$ produces an error of order $h$
which is negligible).  
The next thing to note is that, by another application of~\eqref{sq08:6}, 
\begin{equation}
 \label{sq08:9}
 F\biggpar{\frac{r_0}{\sigma\sqrt{V(\ph+x)}}}
 = F\biggpar{\frac{r_0}{\sigma\sqrt{V(\ph)}}}
 \bigbrak{1 + \Order{x\sigma^{-2\eps}}}\;.
\end{equation}
As a consequence, the integrand in~\eqref{sq08:5} changes by a factor of order
$1$ at most on intervals of order $\sigma^{2\eps}$, and therefore, 
\begin{equation}
 \label{sq08:10}
 \int_{\ph_0}^{\ph_0 + h_2} f_-(\ph)\6\ph 
 \leqs \int_{\ph_0}^{\ph_0 + \sigma^{2\eps}} f_-(\ph)\6\ph \cdot
\Order{\sigma^{2-4\eps}}\;.
\end{equation} 
It follows that 
\begin{align}
\nonumber
 \int_{\ph_0 \mp h_2}^{\ph_1 \pm h_2} f_\mp(\ph)\6\ph 
 ={}& \frac{T_+}{\sqrt{2\pi}} \int_{\ph_0}^{\ph_1} 
 \frac{D_{rr}(1,\ph)\e^{-2\lambda_+ T_+\ph}}{V(\ph)} 
 F \biggpar{\frac{r_0}{\sigma \sqrt{V(\ph)}}} \6\ph
 \\
 &{}\times  
 \bigbrak{1 + \Order{\sigma^\eps} + \Order{\sigma^{1-3\eps}} +
\Order{\sigma^{2-4\eps}}}\;,
 \label{sq08:11}
\end{align}
where the last error term is negligible. 
Finally, the difference of the two terms in~\eqref{sq08:5} involving $\Phi$ is
bounded above by 
\begin{equation}
 \label{sq08:12}
 \frac{2}{\sqrt{2\pi}} \frac{2 H \e^{-\lambda^-T}}{\sigma \sqrt{V(\ph_0-h_2)}}
 \exp\biggset{-\frac{(r_0-H\e^{-\lambda^-T})^2}{2\sigma^2 V(\ph_0-h_2)}}\;.
\end{equation} 
The ratio of~\eqref{sq08:12} and~\eqref{sq08:11} has order $H/r_0 \leqs
\sigma^\eps$. This proves the upper bound in~\eqref{sq06}, and the proof of the
lower bound is analogous. 

\item 	{\it Case $0\leqs r_0 <\sigma^{2-3\eps}$.}
In this case, the comparison with $r^-_t$
becomes useless. Instead of~\eqref{sq08:5} we thus write 
\begin{align}
\nonumber
\bigprobin{r_0}{\ph_\tau\in[\ph_0,\ph_1]}
 &\leqs 1 - \bigprobin{r_0}{\tau^0_+/T_+ \leqs \ph_0 - h_2} +
 \fP(\Omega_1\cap\Omega_2^c) \\
\nonumber
 &= 1 - 2 \Phi\biggpar{-\frac{r_0+H\e^{-\lambda_+T}}{\sigma
\sqrt{V(\ph_0-h_2)}}} + \fP(\Omega_1\cap\Omega_2^c) \\
 &= \Order{r_0/\sigma} + \fP(\Omega_1\cap\Omega_2^c) =
\Order{\sigma^{1-3\eps}}\;.
 \label{sq08:99}
\end{align} 
This proves~\eqref{sq07}. 
\qed
\end{enum}
\renewcommand{\qed}{}
\end{proof}

We now would like to obtain a similar estimate for the hitting distribution
when starting in the QSD $\piOu$ instead of a fixed point $r_0$. Unfortunately,
we do not have much information on $\piOu$. Still, we can draw on the fact that
the distribution of the process conditioned on survival approaches the QSD. 
To do so, we need the existence of a spectral gap for the kernel 
$\Ku$, which will be obtained in Section~\ref{sec_el}. 

\begin{prop}[Bounds on the first-hitting distribution starting from the QSD]
\label{prop_sq2} 
Let $\loneu$ be the second eigenvalue of $\Ku$, and 
assume the spectral gap condition $\abs{\loneu}/\lOu \leqs \rho < 1$ holds  
uniformly in $\sigma$ as $\sigma\to0$.
Fix constants $0<\Delta<\e^{-1/9}$ and $0<\eps<1/3$. There exist constants
$\sigma_0, \gamma, \kappa >0$ such that for all $\sigma < \sigma_0$ and
$\ph_0\in[0,1]$, 
\begin{equation}
\bigprobin{\piOu}{\ph_\tau\in[\ph_0,\ph_0+\Delta]}
 = Z(\sigma) \int_{\ph_0}^{\ph_0+\Delta}
D_{rr}(1,\ph)\e^{-2\lambda_+ T_+\ph}\6\ph 
 \,\brak{1+\Order{\Delta^\beta} + \Order{\Delta^2\abs{\log\Delta}}} 
 \label{sq10}
\end{equation} 
where $Z(\sigma)$ does not depend on $\ph_0$, and $\beta =
2\abs{\log\rho}/(\lambda_+ T_+)$. 
\end{prop}
\begin{proof}
Let $n\in\N$ be such that 
$\Delta^2 < \e^{-2n\lambda T_+} \leqs \e^{2\lambda T_+}\Delta^2$. We let
$I=[\ph_0,\ph_0+\Delta]$, write $n+I$ for the translated interval
$[n+\ph_0,n+\ph_0+\Delta]$, and $\ku_n$ for the density of $\Ku_n$.  
For any initial condition $r_0\in(0,\delta)$, we have 
\begin{align}
\nonumber
\int_0^\delta \ku_n(r_0,r) 
\bigprobin{r}{\ph_\tau\in n+I} \6r
&= \bigprobin{r_0}{\ph_\tau\in 2n+I} \\
\nonumber
&=\int_0^\delta \ku_{2n}(r_0,r) 
\bigprobin{r}{\ph_\tau\in I} \6r \\
\nonumber
&= \int_0^\delta (\lOu)^{2n} N(r_0) \piOu(r) 
\bigprobin{r}{\ph_\tau\in I} \6r \,\brak{1+\Order{\rho^{2n}}} \\
&= (\lOu)^{2n} N(r_0) \bigprobin{\piOu}{\ph_\tau\in I}  
\,\brak{1+\Order{\Delta^\beta}}\;,
\label{sq11:1} 
\end{align}
where $N(r_0)$ is a normalisation, cf.\ \eqref{resrp6}, 
and we have used $\rho^{2n} = \e^{-2n\abs{\log\rho}} \leqs
\rho^{-2}\Delta^\beta$. 
It is thus sufficient to compute the left-hand side for a convenient $r_0$,
which we are going to choose as $r_0=\sigma$. 
By Proposition~\ref{prop_sq1}, we have 
\begin{align}
\nonumber
\int_0^\delta \ku_n(\sigma,r) 
\bigprobin{r}{\ph_\tau\in n+I} \6r
={}& \frac{T_+}{\sqrt{2\pi}} 
\int_{n+I} \frac{D_{rr}(1,\ph)\e^{-2\lambda_+ T_+\ph}}{V(\ph)} J_0(\ph) \6\ph 
\,\brak{1+\Order{\sigma^\gamma}} 
\\
&{}+ \Order{\sigma^{2-3\eps}} + 
\Order{\e^{-\kappa/\sigma^{2\eps}}}\;,
\label{sq11:2} 
\end{align}
where we have split the integral at $r=\sigma^{2-3\eps}$, 
bounded $\Ku_n(\sigma,[0,\sigma^{2-3\eps}])$ by $1$ and introduced
\begin{equation}
 \label{sq11:3}
 J_0(\ph) = \int_{\sigma^{2-3\eps}}^{\delta} 
 \ku_n(\sigma,r) F \biggpar{\frac{r}{\sigma\sqrt{V(\ph)}}}\6r \;.
\end{equation} 
Note that for $\ph\in n+I$, one has $V(\ph) = v_\infty[1 + \Order{\Delta^2}]$.
To complete the proof it is thus sufficient to show that 
$J_0(\ph) = Z_0 \brak{1 + \Order{\Delta^2}\abs{\log\Delta}}$, where $Z_0$ does
not depend on $\ph_0$ and satisfies $Z_0 \geqs \const\, \sigma \Delta^2$. 

We perform the scaling $r=\sigma\sqrt{v_\infty}\,u$ and write 
\begin{equation}
 \label{sq11:4}
 J_0(\ph) = \sigma\sqrt{v_\infty}\,
\int_{\sigma^{1-3\eps}/\sqrt{v_\infty}}^{\delta/\sigma\sqrt{v_\infty}} 
 \ku_n(\sigma,\sigma\sqrt{v_\infty}\,u) F(\mu u)\6u \;,
\end{equation} 
where $\mu = \sqrt{v_\infty/V(\ph)} = 1 + \Order{\Delta^2}$
satisfies $\mu\geqs1$. Let
\begin{equation}
 \label{sq11:4b}
 Z_0 = \sigma\sqrt{v_\infty}\,
\int_{\sigma^{1-3\eps}/\sqrt{v_\infty}}^{\delta/\sigma\sqrt{v_\infty}} 
 \ku_n(\sigma,\sigma\sqrt{v_\infty}\,u) F(u)\6u\;.
\end{equation}
By the first inequality in \eqref{sq08:6}, we immediately have the upper bound 
\begin{equation}
  \label{sq11:5}
 J_0(\ph) \leqs \mu Z_0 \leqs Z_0
 \,\brak{1+\Order{\Delta^2}}\;. 
\end{equation} 
To obtain a matching lower bound, we first show that the integral is dominated
by $u$ of order $1$. Namely, for $0<a<1<b$ of order $1$, 
\begin{align}
\nonumber
 Z_0 
 &\geqs \sigma\sqrt{v_\infty}\, \int_{a/\sqrt{v_\infty}}^{b/\sqrt{v_\infty}} 
 \ku_n(\sigma,\sigma\sqrt{v_\infty}\,u) F(u)\6u \\
 &\geqs \sigma\sqrt{v_\infty}\, C_0 \Ku_n(\sigma,[\sigma a,\sigma b])
 \label{sq11:6}
\end{align} 
where $C_0 = F(a/\sqrt{v_\infty}\,) \wedge F(b/\sqrt{v_\infty}\,)$.
Now Proposition~\ref{prop_lu_lowerbound} implies that 
\begin{equation}
 \label{sq11:7} 
\Ku_n(\sigma,[\sigma a,\sigma b]) 
\geqs C \e^{-2n\lambda T_+}
\geqs C \Delta^2\;.
\end{equation} 
Furthermore, since $3 \sqrt{\abs{\log\Delta}} > 1$, $F$ takes its maximal value
at the lower integration limit, and we have 
\begin{align}
\nonumber
 \int_{3\sqrt{\abs{\log\Delta}}}^{\delta/\sigma\sqrt{v_\infty}} 
 \ku_n(\sigma,\sigma\sqrt{v_\infty}\,u) F(u)\6u 
 &\leqs F \Bigpar{3\sqrt{\abs{\log\Delta}}}
\Ku_n\Bigpar{[3\sigma\sqrt{v_\infty\abs{\log\Delta}},\delta]} \\
\nonumber
&\leqs 3 \Delta^{9/2}\sqrt{\abs{\log\Delta}}  \cdot 1\\
&\leqs \frac{3 \Delta^{5/2}\sqrt{\abs{\log\Delta}}}{C_0C\sigma\sqrt{v_\infty}} 
Z_0\;.
 \label{sq11:8}
\end{align} 
Using again \eqref{sq08:6} and the above estimates, we get the lower bound 
\begin{align}
\nonumber
 J_0(\ph) &\geqs \sigma\sqrt{v_\infty}\,
\int_{\sigma^{1-3\eps}/\sqrt{v_\infty}}^{3\sqrt{\abs{\log\Delta}}} 
 \ku_n(\sigma,\sigma\sqrt{v_\infty}\,u) F(u)
 \e^{-(\mu^2-1)u^2/2}\6u \\ 
\nonumber
 &\geqs 
 \biggbrak{Z_0 - \sigma\sqrt{v_\infty}\,
 \int_{3\sqrt{\abs{\log\Delta}}}^{\delta/\sigma\sqrt{v_\infty}} 
 \ku_n(\sigma,\sigma\sqrt{v_\infty}\,u) F(u)\6u }
 \bigbrak{1 - \Order{\Delta^2\abs{\log\Delta}}} \\
 &= Z_0 \Bigbrak{1 - \bigOrder{\Delta^{5/2} \sqrt{\abs{\log\Delta}}\,}
 - \Order{\Delta^2\abs{\log\Delta}}}\;,
\label{sq11:9}
\end{align} 
which completes the proof.
\end{proof}



\subsection{The principal eigenvalue $\lOs$ and the spectral gap}
\label{ssec_sg}

We consider in this section the system 
\begin{align}
\nonumber
\6r_t &= \bigbrak{-\lambda_- r_t + b_r(r_t,\ph_t)} \6t 
+ \sigma g_r(r_t,\ph_t) \6W_t\;, \\
\6\ph_t &= \Bigbrak{\frac{1}{T_+} + b_\ph(r_t,\ph_t)} \6t 
+ \sigma g_\ph(r_t,\ph_t) \6W_t\;,
\label{ls01} 
\end{align}
describing the dynamics away from the unstable orbit. We have redefined $r$ in
such a way that the stable orbit is now located in $r=0$, and that the unstable
orbit is located in $r=1$.
 
In what follows we consider the Markov chain of the process killed upon
reaching level $1-\delta$ where $\delta\in(1/2,1)$, whose kernel we
denote $\Ks$. The corresponding state space is given by $E=[-L,1-\delta]$ for
some $L\geqs1$. 

\begin{prop}[Lower bound on the principal eigenvalue $\lOs$]
\label{prop_l0s} 
There exists a constant $\kappa>0$ such that 
\begin{equation}
 \label{ls02}
 \lOs \geqs 1 - \e^{-\kappa/\sigma^2}\;.
\end{equation} 
\end{prop}
\begin{proof}
Let $A=[-h,h]$ for some $h>0$. If $h$ is sufficiently small, the stability of
the periodic orbit in $r=0$ implies that any deterministic solution starting in
$(r_0,0)$ with $r_0\in A$ satisfies $\abs{r_1}\leqs h_1<h$ when it reaches the 
line $\ph=1$ at a point $(r_1,1)$. In fact, by slightly enlarging $h_1$ we can ensure that
$\abs{r_t}\leqs h_1<h$ whenever $\ph_t$ is in a small neighbourhood of~$1$. 
Using, for instance,~\cite[Theorem~5.1.18]{BGbook},\footnote{This might seem like slight overkill, but
it works.} one obtains that the random sample path with initial condition $(r_0,0)$
stays, on timescales of order $1$, in a ball around the deterministic solution
with high probability. The probability of leaving the ball is exponentially
small in $1/\sigma^2$. This shows that $K(x,A)$ is exponentially close to $1$
for all $x\in A$ and proves the result, thanks to
Proposition~\ref{prop_estimate_l0}. 
\end{proof}

Note that in the preceding proof we showed that for $A=[-h,h]$, there exist 
constants $C_1, \kappa>0$ such that 
\begin{equation}
 \label{ls03a}
 \sup_{x\in A} K(x,\Ac) + p_{\text{\rm kill}}(A)
 \leqs C_1 
 \e^{-\kappa/\sigma^2}\;.
\end{equation}
The following proposition gives a similar estimate  allowing for initial conditions $x\in E$.

\begin{prop}[Bound on the \lq\lq contraction constant\rq\rq]
\label{prop_contraction} 
Let $A=[-h,h]$. For any $h>0$, there exist $n_1\in\N$ and 
constants $C_1, \kappa>0$ such that 
\begin{equation}
 \label{ls03}
 \gamma^{n_1}(A) \defby 
 \sup_{x\in E} K_{n_1}(x,\Ac) \leqs C_1 
 \e^{-\kappa/\sigma^2}\;.
\end{equation}
\end{prop}
\begin{proof}
Consider a deterministic solution $x^{\det}_t=(r^{\det}_t,\ph^{\det}_t)$ with
initial condition $x_0=(r_0,0)$. The stability of the orbit in $r=0$ implies
that $x^{\det}_t$ will reach a neighbourhood of size $h/2$ of this orbit in a
time $T$ of order $1$. By~\cite[Theorem~5.1.18]{BGbook}, we have for all
$t\geqs0$
\begin{equation}
 \label{ls04}
 \biggprobin{x_0}{\sup_{0\leqs s\leqs t}\norm{x_t-x^{\det}_t}>h_0}
 \leqs C_0(1+t)\e^{-\kappa_0 h_0^2/\sigma^2}
\end{equation} 
for some constants $C_0,\kappa_0>0$. The estimate holds for all $h_0\leqs
h_1/\chi(t)$, where $h_1$ is another constant, and $\chi(t)$ is related to the
local Lyapunov exponent of $x^{\det}_t$. Though $\chi(t)$ may grow exponentially
at first, it will ultimately (that is after a time of order $1$) grow at most
linearly in time, because $x^{\det}_t$ is attracted by the stable orbit. Thus we
have $\chi(T)\leqs 1+ CT$ for some constant $C$. Applying~\eqref{ls04} with
$h_0=h/2$, we find that any sample path which is not killed before time $n_1$
close to $T$ will hit $A$ with a high probability, which yields the result. 
\end{proof}

Let $X_{1}^{x_{1}}$ and $X_{1}^{x_{2}}$ denote the values of the first component
$r_{t}$ of the solution of the SDE~\eqref{ls01} with initial condition $x_{1}$
or $x_{2}$, respectively, at the random time at which $\varphi_{t}$ first
reaches the value~$1$. Note that both processes are driven by the same
realization of the Brownian motion.

\begin{prop}[Bound on the difference of two orbits]
\label{prop_difference} 
There exist constants $h_0, c>0$ and $\rho<1$ such that 
\begin{equation}
 \label{ls05}
 \bigprob{\abs{X^{x_2}_1-X^{x_1}_1} \geqs \rho \abs{x_2-x_1}}
 \leqs \e^{-c/\sigma^2}
\end{equation} 
holds for all $x_1,x_2\in A=[-h_0,h_0]$. 
\end{prop}
\begin{proof}
Let $(\xi_t,\eta_t)$ denote the difference of the two sample paths started in
$(x_1,0)$ and $(x_2,0)$, respectively. It satisfies a system of the form 
\begin{align}
\nonumber
\6\xi_t &= -\lambda_-\xi_t\6t + b_1(\xi_t,\eta_t)\6t + \sigma
g_1(\xi_t,\eta_t)\6W_t\;, \\
\6\eta_t &= b_2(\xi_t,\eta_t)\6t + \sigma
g_2(\xi_t,\eta_t)\6W_t\;,
\label{ls05:1} 
\end{align}
with initial condition $(\xi_0,0)$, where we may assume that $\xi_0=x_2-x_1>0$.
Here $\abs{b_i(\xi,\eta)}\leqs M(\xi^2+\eta^2)$ and $\abs{g_i(\xi,\eta)}\leqs
M(\abs{\xi}+\abs{\eta})$ for $i=1,2$ (remember that both solutions are driven by
the same Brownian motion). Consider the stopping times 
\begin{align}
\nonumber
\tau_\xi &= \inf\setsuch{t>0}{\xi_t > H\e^{-\lambda_{-}t}}\;, \\
\tau_\eta &= \inf\setsuch{t>0}{\abs{\eta_t}>h}\;,
\label{ls05:2} 
\end{align}
where we set $H=\xi_0+h=Lh$. Writing $\eta_t$ in integral form and using
Lemma~\ref{lem_Bernstein}, we obtain 
\begin{equation}
 \label{ls05:3}
 \bigprob{\tau_\eta < \tau_\xi \wedge 2} \leqs \e^{-c_1/\sigma^2}
\end{equation} 
for some constant $c_1>0$, provided $h$ is smaller than some constant depending
only on $M$ and $L$. In a similar way, one obtains 
\begin{equation}
 \label{ls05:4}
 \bigprob{\tau_\xi < \tau_\eta \wedge 2} \leqs \e^{-c_2/\sigma^2}
\end{equation} 
for some constant $c_2>0$, provided $H$ is smaller than some constant depending
only on~$M$. It follows that 
\begin{equation}
 \label{ls05:5}
 \bigprob{\tau_\xi \wedge \tau_\eta \leqs 2} \leqs \e^{-c/\sigma^2}
\end{equation} 
for a $c>0$. Together with the control on the diffusion along $\ph$
(cf.~Proposition~\ref{prop_lu_phi}), we can thus guarantee that both sample
paths have crossed $\ph=1$ before time $2$, at a distance
\begin{equation}
 \label{ls05:6}
 \abs{X^{x_2}_1-X^{x_1}_1} \leqs \e^{-\lambda_-/2} (\xi_0+h)
\end{equation} 
with probability exponentially close to $1$. For any
$\rho\in(\e^{-\lambda_-/2},1)$, we can find $h$ such that the right-hand side
is smaller than $\rho\xi_0 = \rho\abs{x_2-x_1}$. This yields the result. 
\end{proof}

From~\eqref{ls05}, we immediately get 
\begin{equation}
 \label{ls07}
 \bigprob{\abs{X^{x_2}_n-X^{x_1}_n} \geqs \rho^n \abs{x_2-x_1}}
 \leqs n\e^{-c/\sigma^2}
 \qquad  \forall n\geqs 1\;.
\end{equation} 
Fix $a>0$ and let $N$ be the integer stopping time 
\begin{equation}
 \label{ls06} 
 N = N(x_1,x_2) =
\inf\bigsetsuch{n\geqs1}{\abs{X^{x_2}_n-X^{x_1}_n}<a\sigma^2}\;.
\end{equation} 
If $n_0$ is such that $\rho^{n_0}\mbox{diam}(A)\leqs a\sigma^2$, then~\eqref{ls07} implies $\prob{N>n_0}\leqs n_0\e^{-c/\sigma^2}$ whenever $x_1,x_2\in A$. 
Using Proposition~\ref{prop_contraction} and the Markov property, we 
obtain the following improvement.

\begin{prop}[Bound on the hitting time of a small ball]
\label{prop_difference_improved} 
There is a constant $C_2$ such that for any $k\geqs 1$ and all $x_1,x_2\in A$,
we have 
\begin{equation}
 \label{ls08}
 \bigprob{N(x_1,x_2)>kn_0\mid X^{x_1}_{l},X^{x_2}_{l}\in A \,\forall l\leqs k
n_{0}} 
 \leqs 
 \bigpar{C_2\abs{\log(a\sigma^2)}\e^{-\kappa_2/\sigma^2}}^{k}\;,
\end{equation} 
where $\kappa_2= c$. 
\end{prop}
\begin{proof}
By the definition of $n_{0}$ and~\eqref{ls03a}, for any $x_1,x_2\in A$ we have 
\begin{align}
\nonumber
\bigprob{N > n_0 \mid X^{x_1}_{l},X^{x_2}_{l}\in A \,\forall l\leqs n_{0}} 
\leqs{}&{} 
\frac{\displaystyle \prob{N>n_{0}}}{\displaystyle
\prob{X^{x_1}_{l},X^{x_2}_{l}\in A \,\forall l\leqs n_{0}}} \\
\nonumber
{}\leqs{}&{} \frac{\displaystyle n_{0}\e^{-c/\sigma^{2}}}{\displaystyle 1-n_{0}C_{1}\e^{-\kappa/\sigma^{2}}}
\leqs 2n_{0}\e^{-c/\sigma^{2}}\;.
\end{align}
Thus the result follows by applying the Markov property at times which
are multiples of $n_0$, and recalling that $n_0$ has order
$\abs{\log(a\sigma^2)}$. 
\end{proof}

Combining the last estimates with Proposition~\ref{prop_spectralgap}, we
finally obtain the following result.

\begin{theorem}[Spectral gap estimate for $\Ks$]
\label{thm_spectralgap} 
 There exists a constant $c>0$ such that for sufficiently small $\sigma$, the
first eigenvalue of {\rm $\Ks$} satisfies 
 \begin{equation}
  \label{ls10}
  \abs{\lones} \leqs \e^{-c/\logs}
 \end{equation} 
\end{theorem}
\begin{proof}
We take $n=k(n_0+n_1)$, where $k$ will be chosen below. Fix $h>0$ and set $A=(-h,h)$. We apply
Proposition~\ref{prop_spectralgap} for the Markov chain $\Ks_{n}$, conditioned on not leaving $A$, with $m(y)=\inf_{x\in A}k_{n}(x,y)$, which yields
\begin{equation}
 \label{ls10:1}
 \abs{\lambda_1}^n \leqs 
 \max\biggset{2\gamma^n(A) \;,\;
(\lOs)^n L - 1 + \gamma^n(A)
 \frac{(\lOs)^nL}
 {(\lOs)^nL-1}
 \biggbrak{1+\frac{1}{(\lOs)^n -\gamma^n(A)}}
 }\;.
\end{equation} 
Proposition~\ref{prop_contraction} shows that $\gamma^n(A)\leqs \gamma^{n}([-h/2,h/2])$ is exponentially
small, since $n\geqs n_1$. Proposition~\ref{prop_l0s} shows that $(\lOs)^n$
is bounded below by $1-n\e^{-\kappa/\sigma^2}$. It thus remains to estimate
$L$. Proposition~\ref{prop_positivity} shows that 
\begin{equation}
\label{ls10:2} 
 L \leqs \frac{1 + \eta + \sup_{x_{1},x_{2}\in A}
 \prob{N(x_{1},x_{2})>n-1}\e^{C/\sigma^2}}{\inf_{x\in A}K_{n}(x,A)}\;,
\end{equation} 
where the parameter $a$ in the definition of the stopping time  $N$ is
determined by the choice of $\eta$. We thus fix, say, $\eta=1/4$, and  $k =
\intpartplus{C/\kappa_2}+1$. In this way, the numerator in~\eqref{ls10:2} is
exponentially close to $1+\eta$. Since $n$ has order $\logs$, the denominator
$K_{n-1}(x,A)$ is still exponentially close to $1$, by the same argument as in
Proposition~\ref{prop_l0s}. Making $\sigma$ small enough, we can guarantee that
$L-1\leqs 3/8$ and $(\lOs)^n L>1$, and thus $\abs{\lambda_1}^n \leqs 1/2$. The
result thus follows from the fact that $n$ has order $\logs$.  
\end{proof}


\section{Distribution of exit locations}
\label{sec_el}

We can now complete the proof of Theorem~\ref{main_theorem}, which is close in
spirit to the proof of~\cite[Theorem~2.3]{BG7}. 
We fix an initial condition $(r_0,0)$ close to the stable periodic orbit and a
small positive constant $\Delta$. Let 
\begin{equation}
 \label{el01} 
 P_\Delta(\ph) = \bigprobin{r_0,0}{\ph_{\tau_-}\in[\ph_1,\ph_1+\Delta]}
\end{equation} 
be the probability that the first hitting of level $1-\delta$
occurs in the interval $[\ph_1,\ph_1+\Delta]$. 
If we write $\ph_1=k+s$ with $k\in\N$ and $s\in[0,1)$, we have 
by the same argument as the one given in~\eqref{resrp6}, 
\begin{equation}
 \label{el02}
 P_\Delta(k+s) = C(r_0)(\lOs)^k \bigprobin{\piOs}{\ph_{\tau_-}\in[s,s+\Delta]}
 \biggbrak{1+\biggOrder{\biggpar{\frac{\lones}{\lOs}}^k}}\;,
\end{equation} 
where $C(r_0)$ is a normalising constant, $\piOs$ is the quasistationary
distribution for $\Ks$, and 
\begin{equation}
 \label{el03}
 \bigprobin{\piOs}{\ph_{\tau_-}\in[s,s+\Delta]}
 = \int^{1-\delta}_{-L} \piOs(r) \bigprobin{r}{\ph_{\tau_-}\in[s,s+\Delta]}\6r\;.
\end{equation} 
Note that $\piOs$ is concentrated near the stable periodic orbit. By the large
deviation principle and our assumption on the uniqueness of the minimal path
$\gamma_\infty$, $\probin{r}{\ph_{\tau_-}\in[s,s+\Delta]}$ is maximal at the
point $s^*$ where $\gamma_\infty$ crosses the level $1-\delta$, and decays
exponentially fast in $1/\sigma^2$ away from $s^*$. In addition, our spectral
gap estimate Theorem~\ref{thm_spectralgap} shows that the error term in~\eqref{el02} 
has order $\delta$ as soon as $k$ has order
$\abs{\log\sigma}\abs{\log\delta}$. For these $k$ we thus have 
\begin{equation}
 \label{el03b}
 P_\Delta(k+s) = C_1(\lOs)^k \e^{-J(s)/\sigma^2}
\bigbrak{1+\Order{\delta}}\;,
\end{equation}
where $J(s)$ is periodic with a unique minimum per period in $s^*$. The minimal
value $J(s^*)$ is close to the value of the rate function of the optimal
path up to level $1-\delta$. 

Now let us fix an initial condition $(1-\delta,\ph_1)$, and consider the
probability 
\begin{equation}
 \label{el04}
 Q_\Delta(\ph_1,\ph_2) =
\bigprobin{1-\delta,\ph_1}{\ph_{\tau}\in[\ph_2,\ph_2+\Delta]}
\end{equation} 
of first reaching the unstable orbit during the interval $[\ph_2,\ph_2+\Delta]$ when starting at level $1-\delta$.
By the same argument as above, 
\begin{equation}
 \label{el05}
 Q_\Delta(\ph_1,\ell+s) = C_2(\lOu)^\ell
\bigprobin{\piOu}{\ph_{\tau}\in[s,s+\Delta]}
 \biggbrak{1+\biggOrder{\biggpar{\frac{\loneu}{\lOu}}^\ell}}\;.
\end{equation} 
On the other hand, by the large-deviation principle (cf.\
Proposition~\ref{prop_ldp_nonlinear}), we have 
\begin{equation}
 \label{el06}
 Q_\Delta(\ph_1,\ell+s) = D_\ell(s) \exp\biggset{-\frac{1}{\sigma^2}
\Bigbrak{I_\infty + \frac{c(s)}{2} \e^{-2\ell\lambda_+T_+}
\bigbrak{1 + \Order{\e^{-2\ell\lambda_+T_+}} + \Order{\delta}}}}\;,
\end{equation} 
where 
\begin{equation}
 \label{el06b}
 c(s) = \delta^2 \e^{-2\lambda_+T_+s}
 \frac{\hper(s)}{\hper(\ph_1)^2}\;,
\end{equation} 
and the $\sigma$-dependent prefactor satisfies $\lim_{\sigma\to0} \sigma^2\log D_\ell(s) = 0$.
In particular we have
\begin{equation}
 \label{el06c}
 -\lim_{\sigma\to0} \sigma^2 \log Q_\Delta(\ph_1,\ell+s) 
 = I_\infty + \frac{c(s)}{2} \e^{-2\ell\lambda_+T_+}
\bigbrak{1 + \Order{\e^{-2\ell\lambda_+T_+}} + \Order{\delta}}\;.
\end{equation} 
This implies in particular that the kernel $\Ku$ has a spectral gap. Indeed,
assume by contradiction that $\lim_{\sigma\to0} (\loneu/\lOu) = 1$. Using this
in~\eqref{el05}, we obtain that $\sigma^2 \log
Q_\Delta(\ph_1,\ell+s)$ converges to a quantity independent of $\ell$, which
contradicts~\eqref{el06c}. We thus conclude that $\lim_{\sigma\to0}
(\loneu/\lOu) =\rho < 1$. In fact,~\eqref{el06c} suggests that
$\rho\simeq\e^{-2\lambda_+T_+}$, but we will not attempt to prove this here. 
The existence of a spectral gap shows that for all $\ell \gg \abs{\log\delta}$,
we have 
\begin{equation}
 \label{el06d}
 D_\ell(s) = (\lOu)^\ell D^*(s) \bigbrak{1+\Order{\delta}}\;,
\end{equation} 
where $D^*(s) = \lim_{\ell\to\infty} (\lOu)^{-\ell} D_\ell(s)$. 
Proposition~\ref{prop_sq2} and the existence of a spectral gap imply that 
\begin{equation}
 \label{el06E}
 D^*(s) = C_3(\sigma) D_{rr}(1,s) \e^{-2\lambda_+T_+s}
\brak{1+\Order{\Delta^\beta} + \Order{\Delta}}
\end{equation} 
with $\beta>0$. 
Let us now define 
\begin{equation}
 \label{el07}
 \theta(s) = -\frac12 \log c(s)
 = \lambda_+ T_+ s - \frac12\log\biggpar{\delta^2 
 \frac{\hper(s)}{\hper(\ph_1)^2}}
 \;.
\end{equation} 
Note that 
$\theta(s+1)=\theta(s)+\lambda_+T_+$. 
Furthermore, the differential equation satisfied by $\hper(s)$
[cf.~\eqref{equation_for_hper}] and~\eqref{el06b} show that 
\begin{equation}
 \label{el07A}
 \theta'(s) = \frac{D_{rr}(1,s)}{2\hper(s)}
 = \frac{\delta^2}{2\hper(\ph_1)^2}
 D_{rr}(1,s)\e^{-2\lambda_+T_+ s}
\e^{2\theta(s)}\;, 
\end{equation} 
and thus 
\begin{equation}
 \label{el07B}
 D^*(s) = 2C_3 \frac{\hper(\ph_1)^2}{\delta^2} \,
 \theta'(s) \e^{-2\theta(s)}
\brak{1+\Order{\Delta^\beta} + \Order{\Delta}}\;.
\end{equation}
Since $\sigma^{-2}=\e^{2\abs{\log\sigma}}$, 
we can rewrite~\eqref{el06} in the form
\begin{equation}
 \label{el08}
 Q_\Delta(\ph_1,\ell+s) = D_\ell(s) \e^{-I_\infty/\sigma^2}
 \exp\biggset{-\frac12
\e^{-2 \bigbrak{\theta(s) + \ell\lambda_+T_+ - \abs{\log\sigma} 
}}\bigbrak{1+ \Order{\e^{-2\ell\lambda_+T_+}}+\Order{\delta}}}
\;.
\end{equation}
In a similar way, the prefactor $D_\ell(s)$ can be rewritten, for 
$\ell\gg\abs{\log\delta}$, as
\begin{equation}
 \label{el09}
 D_\ell(s) = \sigma^2 \e^{2\theta(s)}D^*(s) 
 \exp\biggset{-2 \Bigbrak{\theta(s) + 
 \ell\lambda_+T_+\brak{1+\Order{\delta}} -
 \abs{\log\sigma}}}\;.
\end{equation} 
Introducing the notation 
\begin{equation}
 \label{el23}
 A(x) = \exp\Bigset{-2x - \frac12 \e^{-2x}}\;,
\end{equation} 
we can thus write 
\begin{equation}
 \label{el10}
 Q_\Delta(\ph_1,\ell+s) =2 \sigma^2 \e^{-I_\infty/\sigma^2}
 \e^{2\theta(s)}D^*(s) 
 A \Bigpar{\theta(s) + 
 \ell\lambda_+T_+ -
 \abs{\log\sigma} + \Order{\delta \ell}}\;.
\end{equation} 
Let us finally consider the probability
\begin{equation}
\label{el20} 
 \bigprobin{r_0,0}{\ph_{\tau}\in[n+s,n+s+\Delta]}\;.
\end{equation}
As in~\cite[Section~5]{BG7}, it can be written as an integral over
$\ph_1$, which can be approximated by the sum 
\begin{equation}
 \label{el21}
 S = \sum_{\ell=1}^{n-1}  P_\Delta(n-\ell+s^*) Q_\Delta(s^*,\ell+s)\;.
\end{equation} 
Note that since $\Ku$ is defined by killing the process when it reaches a
distance $2\delta$ from the unstable periodic orbit, we have to show that the
contribution of paths switching back and forth several times between distance
$\delta$ and $2\delta$ is negligible (cf.~\cite[Section~4.3]{BG7}). This is the
case here as well, in fact we have used the same argument in the proof of
Proposition~\ref{prop_lu_upperbound2}. 
From here on, we can proceed as in~\cite[Section~5.2]{BG7}, to obtain  
\begin{align}
\nonumber
S &= C_1\sigma^2 \e^{-(I_\infty+J(s^*))/\sigma^2} (\lOs)^n
 \e^{2\theta(s)}D^*(s)S_1\\
 &=2C_1C_3 \sigma^2 \frac{\hper(\ph_1)^2}{\delta^2} \,
 \e^{-(I_\infty+J(s^*))/\sigma^2} (\lOs)^n
 \theta'(s) 
S_1 \bigbrak{1+\Order{\Delta^{\beta}}+\Order{\Delta}}
 \;,
 \label{el22}
\end{align} 
where 
\begin{equation}
 \label{el25}
 S_1 = 
 \sum_{\ell=-\infty}^\infty (\lOs)^{-\ell} 
 A\Bigpar{\theta(s) + \ell\lambda_+T_+ - \abs{\log\sigma} + \Order{\delta} }
\bigbrak{1+\Order{\delta\abs{\log\delta}}}\;. 
\end{equation} 
The main point is to note that only indices $\ell$ in a window of order 
$\abs{\log{\delta}}$ contribute to the sum, and that
$A(x+\eps)=A(x)[1+\Order{\eps}]$ for $x\geqs0$. Also, since $\lOs$ is
exponentially close to $1$, the factor $(\lOs)^{-\ell}$ can be replaced by $1$,
with an error which is negligible compared to $\delta\abs{\log\delta}$.
Extending the bounds to $\pm\infty$ only generates a small error. 
Now~\eqref{el22} and~\eqref{el25} yield the main result, after performing the
change of variables $\ph\mapsto\theta(\ph)$, and replacing $\beta$ by its
minimum with $1$.  


\appendix

\section{A Bernstein--type estimate}

\begin{lemma}[{\cite[Thm.~37.8]{RogersWilliams}}]
\label{lem_Bernstein}
Consider the martingale 
\begin{equation}
 M_t = \int_0^t g(X_s,s)\6W_s\;,
\end{equation} 
where $X_s$ is adapted to the filtration generated by $W_s$. Assume 
\begin{equation}
 g(X_s,s)\transpose{g(X_s,s)} \leqs G(s)^2
\end{equation} 
almost surely and that 
\begin{equation}
 V(t) = \int_0^t G(s)^2\6s < \infty\;.
\end{equation} 
Then
\begin{equation}
 \bigprob{\sup_{0\leqs s\leqs t} M_s > L} \leqs \e^{-L^2/2V(t)}
\end{equation} 
holds for all $L>0$. 
\end{lemma}


{\small
\bibliography{../../BDGK}

\newcommand{\etalchar}[1]{$^{#1}$}
\def\cprime{$'$}
\providecommand{\bysame}{\leavevmode\hbox to3em{\hrulefill}\thinspace}
\providecommand{\MR}{\relax\ifhmode\unskip\space\fi MR }
\providecommand{\MRhref}[2]{%
  \href{http://www.ams.org/mathscinet-getitem?mr=#1}{#2}
}
\providecommand{\href}[2]{#2}
\begin{thebibliography}{CGLM13}

\bibitem[BAKS84]{BenArous_Kusuoka_Stroock_1984}
G{\'e}rard Ben~Arous, Shigeo Kusuoka, and Daniel~W. Stroock, \emph{The
  {P}oisson kernel for certain degenerate elliptic operators}, J. Funct. Anal.
  \textbf{56} (1984), no.~2, 171--209.

\bibitem[Ber13]{Berglund_Kramers}
Nils Berglund, \emph{{K}ramers' law: Validity, derivations and
  generalisations}, {\tt arXiv:1106.5799}. To appear in Markov Process. Related
  Fields, 2013.

\bibitem[BG04]{BG7}
Nils Berglund and Barbara Gentz, \emph{On the noise-induced passage through an
  unstable periodic orbit~{I}: {T}wo-level model}, J.~{Statist.} {Phys.}
  \textbf{114} (2004), 1577--1618.

\bibitem[BG05]{BG9}
\bysame, \emph{Universality of first-passage and residence-time distributions
  in non-adiabatic stochastic resonance}, Europhys. Letters \textbf{70} (2005),
  1--7.

\bibitem[BG06]{BGbook}
\bysame, \emph{Noise-induced phenomena in slow--fast dynamical systems. a
  sample-paths approach}, Probability and its Applications, Springer-Verlag,
  London, 2006.

\bibitem[BG09]{BG_neuro09}
\bysame, \emph{Stochastic dynamic bifurcations and excitability}, Stochastic
  Methods in Neuroscience (Carlo Laing and Gabriel Lord, eds.), Oxford
  University Press, 2009, pp.~64--93.

\bibitem[BGK12]{BGK12}
Nils Berglund, Barbara Gentz, and Christian Kuehn, \emph{Hunting {F}rench ducks
  in a noisy environment}, J. Differential Equations \textbf{252} (2012),
  no.~9, 4786--4841.

\bibitem[Bir57]{Birkhoff1957}
Garrett Birkhoff, \emph{Extensions of {J}entzsch's theorem}, Trans. Amer. Math.
  Soc. \textbf{85} (1957), 219--227.

\bibitem[BL12]{BerglundLandon}
Nils Berglund and Damien Landon, \emph{Mixed-mode oscillations and interspike
  interval statistics in the stochastic {F}itz{H}ugh-{N}agumo model},
  Nonlinearity \textbf{25} (2012), no.~8, 2303--2335.

\bibitem[CGLM13]{CerouGuyaderLelievreMalrieu12}
Fr\'ed\'eric C\'erou, Arnaud Guyader, Tony Leli\`evre, and Florent Malrieu,
  \emph{On the length of one-dimensional reactive paths}, ALEA, Lat. Am. J.
  Probab. Math. Stat. \textbf{10} (2013), no.~1, 359--389.

\bibitem[CZ87]{Cranston_Zhao_87}
M.~Cranston and Z.~Zhao, \emph{Conditional transformation of drift formula and
  potential theory for {$\frac12\Delta +b(\cdot)\cdot\nabla$}}, Comm. Math.
  Phys. \textbf{112} (1987), no.~4, 613--625.

\bibitem[Dah77]{Dahlberg1977}
Bj{\"o}rn E.~J. Dahlberg, \emph{Estimates of harmonic measure}, Arch. Rational
  Mech. Anal. \textbf{65} (1977), no.~3, 275--288.

\bibitem[Day83]{Day1}
Martin~V. Day, \emph{On the exponential exit law in the small parameter exit
  problem}, Stochastics \textbf{8} (1983), 297--323.

\bibitem[Day90a]{Day5}
Martin Day, \emph{Large deviations results for the exit problem with
  characteristic boundary}, J. Math. Anal. Appl. \textbf{147} (1990), no.~1,
  134--153.

\bibitem[Day90b]{Day7}
Martin~V. Day, \emph{Some phenomena of the characteristic boundary exit
  problem}, Diffusion processes and related problems in analysis, Vol.~I
  (Evanston, IL, 1989), Progr. Probab., vol.~22, Birkh\"auser Boston, Boston,
  MA, 1990, pp.~55--71.

\bibitem[Day92]{Day3}
\bysame, \emph{Conditional exits for small noise diffusions with characteristic
  boundary}, Ann. Probab. \textbf{20} (1992), no.~3, 1385--1419.

\bibitem[Day94]{Day6}
\bysame, \emph{Cycling and skewing of exit measures for planar systems}, Stoch.
  Stoch. Rep. \textbf{48} (1994), 227--247.

\bibitem[Day96]{Day4}
\bysame, \emph{Exit cycling for the van der {P}ol oscillator and quasipotential
  calculations}, J. Dynam. Differential Equations \textbf{8} (1996), no.~4,
  573--601.

\bibitem[DG13]{Ditlevsen_Greenwood_12}
Susanne Ditlevsen and Priscilla Greenwood, \emph{The {M}orris-{L}ecar neuron
  model embeds a leaky integrate-and-fire model}, Journal of Mathematical
  Biology \textbf{67} (2013), no.~2, 239--259.

\bibitem[DV76]{DonskerVaradhan76}
M.~D. Donsker and S.~R.~S. Varadhan, \emph{On the principal eigenvalue of
  second-order elliptic differential operators}, Comm. Pure Appl. Math.
  \textbf{29} (1976), no.~6, 595--621.

\bibitem[Eyr35]{Eyring}
H.~Eyring, \emph{The activated complex in chemical reactions}, Journal of
  Chemical Physics \textbf{3} (1935), 107--115.

\bibitem[Fen79]{Fenichel}
Neil Fenichel, \emph{Geometric singular perturbation theory for ordinary
  differential equations}, J. Differential Equations \textbf{31} (1979), no.~1,
  53--98.

\bibitem[Fre03]{Fredholm_1903}
Ivar Fredholm, \emph{Sur une classe d'\'equations fonctionnelles}, Acta Math.
  \textbf{27} (1903), no.~1, 365--390.

\bibitem[FW98]{FW}
M.~I. Freidlin and A.~D. Wentzell, \emph{Random perturbations of dynamical
  systems}, second ed., Springer-Verlag, New York, 1998.

\bibitem[GHJM98]{GHM}
Luca Gammaitoni, Peter H\"anggi, Peter Jung, and Fabio Marchesoni,
  \emph{Stochastic resonance}, Rev.\ Mod.\ Phys. \textbf{70} (1998), 223--287.

\bibitem[GR09]{Getfert_Reimann_2009}
Sebastian Getfert and Peter Reimann, \emph{Suppression of thermally activated
  escape by heating}, Phys. Rev. E \textbf{80} (2009), 030101.

\bibitem[GR10]{Getfert_Reimann_2010}
\bysame, \emph{Thermally activated escape far from equilibrium: A unified
  path-integral approach}, Chemical Physics \textbf{375} (2010), no.~2–3, 386
  -- 398.

\bibitem[GT84]{GrahamTel84}
R.~Graham and T.~T\'el, \emph{Existence of a potential for dissipative
  dynamical systems}, Phys.\ Rev.~Letters \textbf{52} (1984), 9--12.

\bibitem[GT85]{GrahamTel85}
\bysame, \emph{Weak-noise limit of {F}okker--{P}lanck models and
  nondifferentiable potentials for dissipative dynamical systems}, Phys.\
  Rev.~A \textbf{31} (1985), 1109--1122.

\bibitem[GT01]{Gilbarg_Trudinger}
David Gilbarg and Neil~S. Trudinger, \emph{Elliptic partial differential
  equations of second order}, Classics in Mathematics, Springer-Verlag, Berlin,
  2001, Reprint of the 1998 edition.

\bibitem[Jen12]{Jentzsch1912}
Jentzsch, \emph{{\"U}ber {I}ntegralgleichungen mit positivem {K}ern}, {J}. f.
  d. reine und angew. {M}ath. \textbf{141} (1912), 235--244.

\bibitem[JK82]{JerisonKenig82}
David~S. Jerison and Carlos~E. Kenig, \emph{Boundary behavior of harmonic
  functions in nontangentially accessible domains}, Adv. in Math. \textbf{46}
  (1982), no.~1, 80--147.

\bibitem[Kra40]{Kramers}
H.~A. Kramers, \emph{Brownian motion in a field of force and the diffusion
  model of chemical reactions}, Physica \textbf{7} (1940), 284--304.

\bibitem[ML81]{MorrisLecar81}
C.~Morris and H.~Lecar, \emph{Voltage oscillations in the barnacle giant muscle
  fiber}, Biophys. J. (1981), 193--213.

\bibitem[MS96]{MS4}
Robert~S. Maier and D.~L. Stein, \emph{Oscillatory behavior of the rate of
  escape through an unstable limit cycle}, Phys. Rev. Lett. \textbf{77} (1996),
  no.~24, 4860--4863.

\bibitem[MS97]{MS1}
Robert~S. Maier and Daniel~L. Stein, \emph{Limiting exit location distributions
  in the stochastic exit problem}, SIAM J. Appl. Math. \textbf{57} (1997),
  752--790.

\bibitem[MS01]{MS2}
Robert~S. Maier and D.~L. Stein, \emph{Noise-activated escape from a sloshing
  potential well}, Phys. Rev. Lett. \textbf{86} (2001), no.~18, 3942--3945.

\bibitem[PRK01]{PRK}
Arkady Pikovsky, Michael Rosenblum, and J{\"u}rgen Kurths,
  \emph{Synchronization, a universal concept in nonlinear sciences}, Cambridge
  Nonlinear Science Series, vol.~12, Cambridge University Press, Cambridge,
  2001.

\bibitem[RE89]{RinzelErmentrout}
J.~Rinzel and B.~Ermentrout, \emph{Analysis of neural excitability and
  oscillations}, Methods of Neural Modeling: From Synapses to Networks (C.~Koch
  and I.~Segev, eds.), MIT Press, 1989, pp.~135--169.

\bibitem[RW00]{RogersWilliams}
L.~C.~G. Rogers and David Williams, \emph{Diffusions, {M}arkov processes, and
  martingales. {V}ol. 2}, Cambridge Mathematical Library, Cambridge University
  Press, Cambridge, 2000, It\^o calculus, Reprint of the second (1994) edition.

\bibitem[SVJ66]{Seneta_VereJones_1966}
E.~Seneta and D.~Vere-Jones, \emph{On quasi-stationary distributions in
  discrete-time {M}arkov chains with a denumerable infinity of states}, J.
  Appl. Probability \textbf{3} (1966), 403--434.

\bibitem[TKY{\etalchar{+}}06]{Tsumoto_etal_2006}
Kunichika Tsumoto, Hiroyuki Kitajima, Tetsuya Yoshinaga, Kazuyuki Aihara, and
  Hiroshi Kawakami, \emph{Bifurcations in morris–lecar neuron model},
  Neurocomputing \textbf{69} (2006), no.~4–6, 293 -- 316.

\bibitem[TP04]{Tateno_Pakdaman_2004}
Takashi Tateno and Khashayar Pakdaman, \emph{Random dynamics of the
  {M}orris-{L}ecar neural model}, Chaos \textbf{14} (2004), no.~3, 511--530.

\bibitem[Tuc75]{Tuckwell75}
Henry~C. Tuckwell, \emph{Determination of the inter-spike times of neurons
  receiving randomly arriving post-synaptik potentials}, Biol. Cybernetics
  \textbf{18} (1975), 225--237.

\bibitem[Tuc89]{Tuckwell}
\bysame, \emph{Stochastic processes in the neurosciences}, SIAM, Philadelphia,
  PA, 1989.

\bibitem[Yag47]{Yaglom56}
A.~M. Yaglom, \emph{Certain limit theorems of the theory of branching random
  processes}, Doklady Akad. Nauk SSSR (N.S.) \textbf{56} (1947), 795--798.

\end{thebibliography}
\bibliographystyle{amsalpha}               
}


\newpage
\tableofcontents


\vfill

\bigskip\bigskip\noindent
{\small
Nils Berglund \\ 
Universit\'e d'Orl\'eans, Laboratoire {\sc Mapmo} \\
{\sc CNRS, UMR 7349} \\
F\'ed\'eration Denis Poisson, FR 2964 \\
B\^atiment de Math\'ematiques, B.P. 6759\\
45067~Orl\'eans Cedex 2, France \\
{\it E-mail address: }{\tt nils.berglund@univ-orleans.fr}

\bigskip\noindent
Barbara Gentz \\ 
{\sc Faculty of Mathematics, University of Bielefeld} \\
P.O. Box 10 01 31, 33501~Bielefeld, Germany \\
{\it E-mail address: }{\tt gentz@math.uni-bielefeld.de}

}



\end{document}